\documentclass[onefignum,onetabnum,pagebackref]{siamart220329}

\usepackage{lipsum}
\usepackage{amsfonts}
\usepackage{graphicx}
\usepackage{epstopdf}
\usepackage{tikz}
\usepackage{hyperref}
\usepackage{xurl}
\usepackage[version=4]{mhchem}
\usetikzlibrary{positioning}
\ifpdf
  \DeclareGraphicsExtensions{.eps,.pdf,.png,.jpg}
\else
  \DeclareGraphicsExtensions{.eps}
\fi

\usepackage{enumitem}
\setlist[enumerate]{leftmargin=.5in}
\setlist[itemize]{leftmargin=.5in}

\usepackage{amsmath, amssymb, mathtools, stmaryrd}
\SetSymbolFont{stmry}{bold}{U}{stmry}{m}{n}
\usepackage{amsfonts,mathrsfs, mleftright,booktabs,soul}
\usepackage{subcaption}
\usepackage{xspace}
\usepackage{algorithm, algpseudocode, algorithmicx}
\usepackage{backref}

\newsiamthm{fact}{Fact}
\newsiamthm{question}{Question}

\headers{Accelerated randomly pivoted Cholesky}{E. N. Epperly, J. A. Tropp, and R. J. Webber}

\title{Embrace rejection: 
Kernel matrix approximation \\ by 
accelerated 
randomly pivoted Cholesky\thanks{Date: October 4, 2024.  Revised: April 1, 2025
\funding{ENE acknowledges support from the  U.S. Department of Energy, Office of Science, Office of Advanced Scientific Computing Research,  Department of Energy Computational Science Graduate Fellowship under Award Number DE-SC0021110.
JAT and RJW acknowledge support from the Office of Naval Research through Awards N00014-18-1-2363 and N00014-24-1-2223, from the National Science Foundation through FRG Award 1952777, and from the Caltech Carver Mead New Adventures Fund.}}}

\author{Ethan N. Epperly\thanks{Division of Computing and Mathematical Sciences, California Institute of Technology, Pasadena, CA 91125 USA (\email{eepperly@caltech.edu}, \email{jtropp@caltech.edu}).}
\and Joel A. Tropp\footnotemark[2] \and Robert J. Webber\thanks{Department of Mathematics, University of California San Diego, La Jolla, CA 92093 USA (\email{rwebber@ucsd.edu}).}}

\usepackage{amsopn}

\newcommand{\real}{\mathbb{R}}
\newcommand{\complex}{\mathbb{C}}

\renewcommand{\left}{\mleft}
  \renewcommand{\right}{\mright}
\usepackage{tcolorbox}
\newcommand{\boxedtext}[1]{\begin{tcolorbox}[colback=white,colframe=black,width=\columnwidth,boxsep=2pt,arc=4pt]
    #1
\end{tcolorbox}}

\DeclareMathOperator{\tr}{tr}
\DeclareMathOperator{\diag}{diag}
\newcommand{\mat}[1]{\boldsymbol{#1}}
\renewcommand{\vec}[1]{\boldsymbol{#1}}

\newcommand{\norm}[1]{\left\| #1 \right\|}

\DeclareMathOperator{\rank}{rank}

\newcommand{\expmat}[1]{\begin{bmatrix} #1 \end{bmatrix}}
\newcommand{\twobytwo}[4]{\expmat{#1 & #2 \\ #3 & #4}}

\newcommand{\onebytwo}[2]{\expmat{#1 & #2}}

\newcommand{\Id}{\mathbf{I}}

\DeclareMathOperator{\expect}{\mathbb{E}}
\DeclareMathOperator{\prob}{\mathbb{P}}

\newcommand{\order}{\mathcal{O}}

\newcommand{\set}[1]{\mathsf{#1}}

\renewcommand{\Re}{\mathrm{Re}}

\renewcommand{\hat}[1]{\widehat{#1}}
\renewcommand{\tilde}[1]{\widetilde{#1}}

\usepackage{listings}
\usepackage{color} 
\definecolor{mygreen}{RGB}{28,172,0} 
\definecolor{mylilas}{RGB}{170,55,241}
\newcommand{\QR}{\textsf{QR}\xspace}
\crefname{lstlisting}{Program}{Programs}
\Crefname{lstlisting}{Program}{Programs}

\renewcommand{\epsilon}{\varepsilon}

\newcommand{\RPCholesky}{\textsc{RP\-Chol\-esky}\xspace}

\newcommand{\RejectChol}{\textsc{RejectionSampleSubmatrix}\xspace}
\usepackage{blkarray}

\newcommand{\Ahat}{\smash{\mat{\hat{A}}}}
\newcommand{\Bhat}{\smash{\mat{\hat{B}}}}

\renewcommand*{\backref}[1]{}
\renewcommand*{\backrefalt}[4]{%
	\ifcase #1 %
	(No citations.)
	\or
	(Cited on page #2.)
	\else
	(Cited on pages #2.)
	\fi
}

\usepackage{xcolor}

\usepackage[sort]{cite}

\ifpdf
\hypersetup{
	pdftitle={Accelerated randomly pivoted Cholesky},
	pdfauthor={E. N. Epperly, J. A. Tropp, and R. J. Webber}
}
\fi

\begin{document}
	
\maketitle
	
\begin{abstract}
Randomly pivoted Cholesky (\RPCholesky) is an algorithm for constructing a low-rank approximation of a positive-semidefinite matrix using a small number of columns.
This paper develops an accelerated version of \RPCholesky that employs block matrix computations and rejection sampling to efficiently simulate the execution of the original algorithm.
For the task of approximating a kernel matrix, the accelerated algorithm can run over $40\times$ faster.
The paper contains implementation details, theoretical guarantees,
experiments on benchmark data sets, and an application to computational chemistry.
\end{abstract}
	
\begin{keywords}
kernel method, low-rank approximation, Nyström approximation, Cholesky decomposition
\end{keywords}
	
\begin{AMS}
65F55, 65C99, 68T05
\end{AMS}
	
\section{Introduction} \label{sec:intro}

This paper treats a core task in computational linear algebra:
\boxedtext{
\flushleft
\textbf{Low-rank psd approximation:} Given a positive-semidefinite (psd) matrix $\mat{A} \in \complex^{N\times N}$ and a rank parameter $1\le k < N$, compute a matrix $\mat{F} \in \complex^{N\times k}$ such that $\mat{A} \approx \mat{F}\mat{F}^*$.
}
\noindent Randomly pivoted Cholesky (\RPCholesky) is a randomized variant of the pivoted partial Cholesky method that stands among the best algorithms for constructing a rank-$k$ psd approximation from a set of $k$ columns~\cite{CETW23}.
The present work introduces \textbf{accelerated} \RPCholesky, a new algorithm that runs up to $40\times$ faster than \RPCholesky, while producing approximations with the same quality.  The accelerated method simulates the behavior of the simpler method by means of block-matrix computations and rejection sampling.
Its benefits are most pronounced when the matrix dimension $N$ and the approximation rank $k$ are both moderately large (say, $N \geq 10^5$ and $k \geq 10^3$), a regime where existing algorithms struggle.

\subsection{Access models for matrices}

When designing an algorithm for low-rank approximation of a psd matrix $\mat{A}$,
one must consider how the algorithm interacts with the matrix.
For example, in the \emph{matvec access model}, the primitive operation is
the matrix--vector product $\vec{v} \mapsto \mat{A} \vec{v}$,
and the algorithm designer aims to construct the low-rank approximation
with the minimum number of matrix--vector products~\cite{TW23,Woo14a,MDM+22}.
The matvec access model can arise from the discretization
of differential and integral operators, among other applications.

In the \textit{entry access model}, the primitive operation is
the evaluation of the matrix entry $\mat{A}(i,j)$
for any pair $(i, j)$ of row and column indices.
The papers~\cite{WS00,FS01,DM05,MM17,CETW23,rudi2018fast} contain efficient
algorithms for low-rank psd approximation in the entry access model.
These methods are appealing because they only access and manipulate
a small fraction of the entries in the target matrix.

The entry access model does not always provide a satisfactory abstraction
for scientific computing and machine learning because it fails to account for computational efficiencies observed in practice by accessing the matrix in a block-wise fashion.
To remedy this deficiency, this paper works in a paradigm that will be called the \emph{submatrix access model}.
The primitive operation
is extracting a submatrix:
\begin{equation} \label{eq:submatrix-access}
    \mat{A}(\set{S}_1, \set{S}_2) = [\mat{A}(i,j)]_{i \in \set{S}_1, j \in \set{S}_2} \quad \text{for } \set{S}_1, \set{S}_2 \subseteq \{1, \ldots, N\}.
\end{equation}
This access model describes the setting
where the cost of generating a moderately large submatrix
is comparable with the cost of evaluating a single matrix entry.
As the next subsection explains, the submatrix access model reflects the computational
challenge that arises from kernel methods in machine learning.

\subsection{Kernel matrices and the submatrix access model} \label{sec:kernels-submatrix}

Recall that a positive-definite kernel function $\kappa:\real^d \times \real^d \to \real$
can be used to define a similarity metric between pairs of data points~\cite{SS02}.
Popular kernel functions include the Gaussian and $\ell_1$ Laplace kernels:
\begin{align}
    \kappa(\vec{x},\vec{x}') &= \exp \Bigl( - \tfrac{1}{2\sigma^2} \norm{\vec{x} - \vec{x}'}^2 \Bigr),& &\text{\textcolor{gray}{[Gaussian]}} \label{eq:gaussian} \\
    \kappa(\vec{x},\vec{x}') &= \exp \Bigl( - \tfrac{1}{\sigma}\sum\nolimits_{i=1}^d |\vec{x}(i) - \vec{x}'(i)|\Bigr). &&\text{\textcolor{gray}{[$\ell_1$ Laplace]}} \label{eq:l1_laplace}
\end{align}
Here, $\|\cdot\|$ is the $\ell_2$ norm and $\sigma >0$ is the bandwidth, which sets the scale of interactions.
Given data points $\vec{x}_1, \dots, \vec{x}_N \in \real^d$,
the kernel matrix $\mat{A} = [\kappa(\vec{x}_i, \vec{x}_j)]_{1 \leq i,j\leq N}$
is a psd matrix that tabulates the similarity between each pair of points.
Kernel methods perform dense linear algebra with the kernel matrix
to accomplish machine learning tasks, such as regression and clustering.

On modern computers, the runtime of numerical algorithms is often dominated by time moving data between cache and memory rather than time performing arithmetic operations \cite[pp.~13, 581--582]{DFF+03}.
Each time columns of a kernel matrix are generated, whether a single column or a large block of $b\gg1$ columns, all $N$ data points $\{\vec{x}_i\}_{i=1}^N$ must be moved from memory to cache.
When data movement dominates the runtime, generating columns in blocks of size $b$ can be up to $b \times$ as fast as generating columns one-by-one.
See \cref{sec:why-blocking} for discussion of other reasons why blocking can increase the speed of matrix algorithms.

\begin{figure}[t]
    \centering

    \begin{subfigure}{0.48\textwidth}
        \centering\hspace{.5cm}Gaussian
        
        \vspace{0.5em}
        \includegraphics[width = \textwidth]{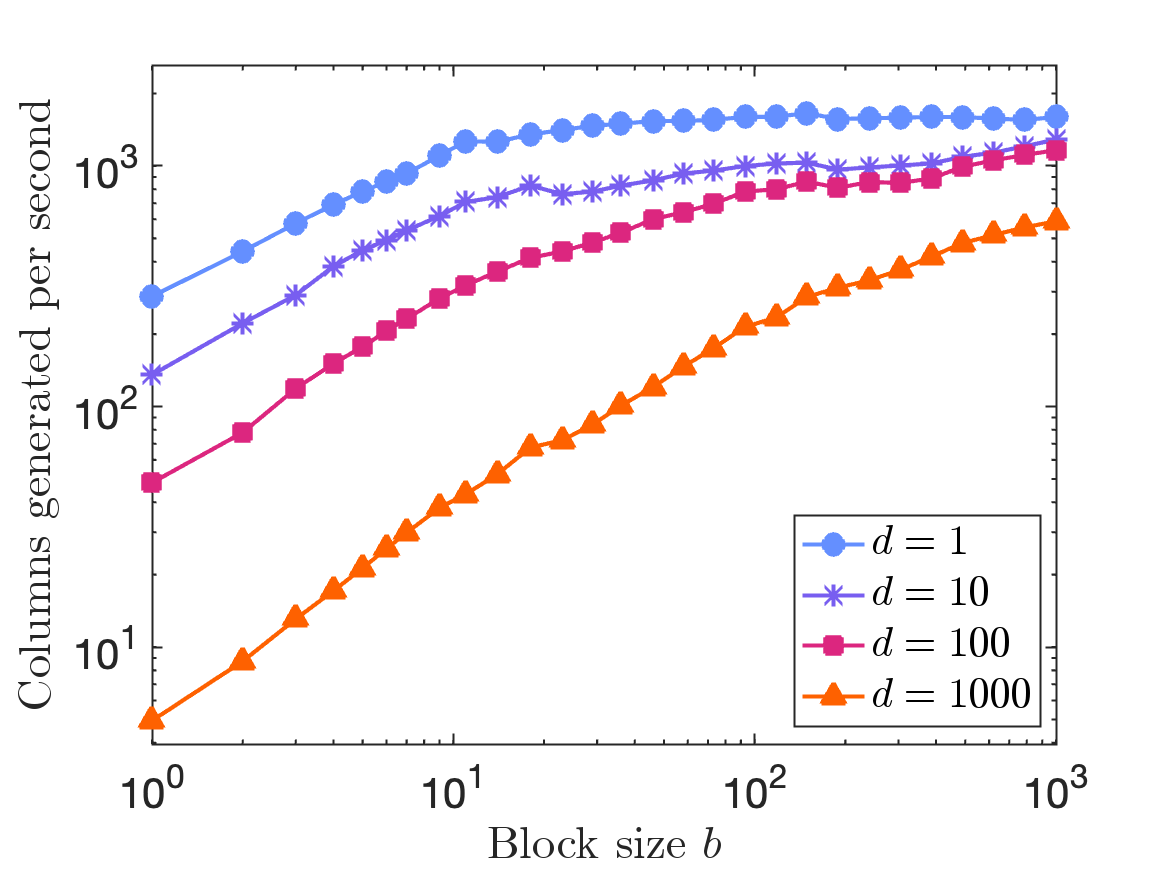}
    \end{subfigure}
    \begin{subfigure}{0.48\textwidth}
        \centering
        \hspace{.5cm}$\ell_1$ Laplace
        
        \vspace{0.5em}
        \includegraphics[width = \textwidth]{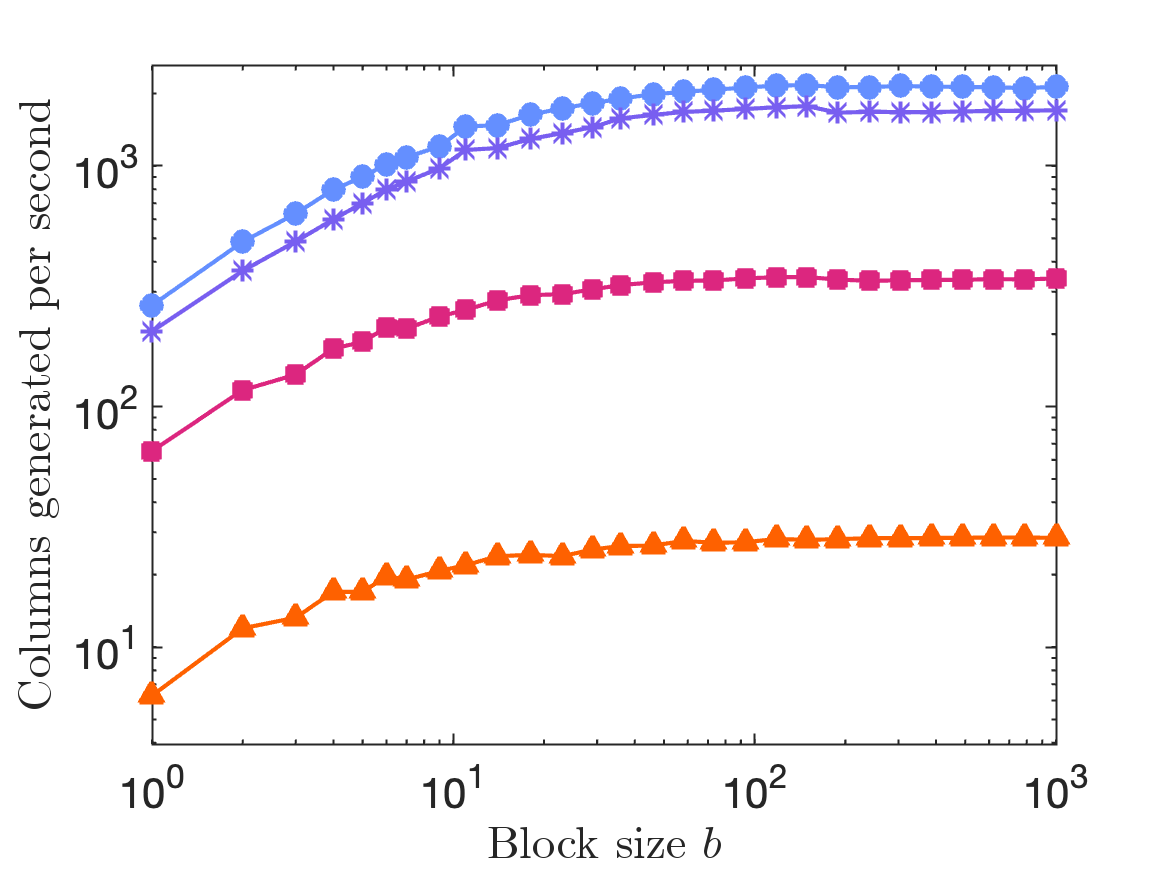}
    \end{subfigure}

    \caption{Columns generated per second for Gaussian (left) and $\ell_1$ Laplace (right) kernel matrices with bandwidth $\sigma = 1$.
    The data consists of $N=10^5$ standard Gaussian points with dimension $d \in \{ 1, 10, 100, 1000 \}$.}
    \label{fig:generation}
\end{figure}

\Cref{fig:generation} provides further empirical evidence that the submatrix access model
is the appropriate abstraction for kernel matrix computations.
The figure charts the number of columns formed per second when evaluating blocks of $b$ columns in a $10^5 \times 10^5$ kernel matrix, with varying block sizes $1 \le b \le 10^3$.
The per-column efficiency increases with $b$, and it saturates once the block size $b$ is large enough to maximally reap the benefits of block-wise access.
This critical block size and the maximum speedup is sensitive both to the dimension $d$ of the data and the choice of kernel.
In the most extreme case (Gaussian kernel with $d=1000$), generating columns in blocks of size $b = 1000$ is $100\times$ as efficient as generating columns one at a time.

For moderately large data sets (say, $N \ge 10^5$), it is prohibitively
expensive to generate and store the full kernel matrix.
Fortunately, the kernel matrix often has rapidly decaying eigenvalues,
so it can be approximated by a low-rank matrix constructed from a subset of the columns.
When properly implemented with block-wise access, this approximation makes it possible to scale kernel methods to billions of data points~\cite{MCRR20}.
The resulting machine learning methods can achieve cutting-edge
speed and interpretability for tasks in scientific machine learning~\cite{aristoff2024fast}.

\subsection{Efficient low-rank approximation in the submatrix access model}

Given the large speedups observed in \cref{fig:generation},
it is natural to ask a question:
\boxedtext{
\flushleft
What algorithm for low-rank psd approximation is most efficient in the submatrix access model?
}
\noindent This paper answers the question by proposing the accelerated \RPCholesky algorithm.
To contextualize this new method, the rest of this subsection reviews simple \RPCholesky (\cref{sec:rpcholesky}) and block \RPCholesky (\cref{sec:block-rpcholesky}).
Last,
\cref{sec:accelerated_intro} takes a first look at accelerated \RPCholesky.

\subsubsection{Simple \RPCholesky} \label{sec:rpcholesky}
Randomly pivoted Cholesky (\RPCholesky) is an algorithm for low-rank psd approximation that is designed for the entry access model~\cite{CETW23}.
After extracting the diagonal of the matrix, the algorithm forms a rank-$k$ approximation
by adaptively evaluating and manipulating $k$ columns of the matrix.

Starting with the trivial approximation $\vphantom{\Ahat^{(0)}}\Ahat^{(0)} \coloneqq \mat{0}$ and the initial residual $\mat{A}^{(0)}\coloneqq \mat{A}$, \RPCholesky builds up a rank-$k$ approximation $\smash{\Ahat}^{(k)}$ by taking the following steps. 
For $i=0,\ldots,k-1$,
\begin{enumerate}
    \item \textbf{Choose a random pivot.} Select a random pivot $s_{i+1} \in \{1,\ldots,N\}$ according to the probability distribution
    \begin{equation}
    \label{eq:first_step}
        \prob\{s_{i+1} = j\} = \frac{\mat{A}^{(i)}(j,j)}{\tr \mat{A}^{(i)}}.
    \end{equation}
    \item \textbf{Update.} Evaluate the column $\mat{A}^{(i)}(:, s_{i+1})$ of the residual indexed by the selected pivot $s_{i+1}$.  Update the approximation and the residual:
    \begin{align*}
        \smash{\Ahat}^{(i+1)}
        &\coloneqq \Ahat^{(i)}
        + \frac{\mat{A}^{(i)}(:,s_{i+1})\mat{A}^{(i)}(s_{i+1},:)}{\mat{A}^{(i)}(s_{i+1},s_{i+1})}; \\
        \mat{A}^{(i+1)}
        &\coloneqq \mat{A}^{(i)}
        - \frac{\mat{A}^{(i)}(:,s_{i+1})\mat{A}^{(i)}(s_{i+1},:)}{\mat{A}^{(i)}(s_{i+1},s_{i+1})}.
    \end{align*}
\end{enumerate}
In this algorithm and those that follow, ``pivot'' refers to the index of the selected column. Pivots are chosen by means of an adaptive, randomized selection rule.
Additionally, $\mat{A}(\cdots)$ indicates access of $\mat{A}$ by MATLAB-style indexing notation; see \cref{sec:notation} for a discussion of notation.

With an optimized implementation (\cref{alg:rpcholesky}), \RPCholesky produces the approximation $\Ahat = \smash{\Ahat}^{(k)}$ in factored form after $\order(k^2N)$ arithmetic operations.
The history of \RPCholesky and related algorithms \cite{CETW23,DRVW06,DV06,MW17} is surveyed in \cite[sec.~3]{CETW23}.

\subsubsection{Block \RPCholesky}\label{sec:block-rpcholesky}

\RPCholesky is less efficient than it could be because the algorithm evaluates columns one by one,
failing to take advantage of the submatrix access model.
For this reason, large-scale machine learning applications \cite{DEF+23,aristoff2024fast} often employ the block \RPCholesky algorithm \cite{CETW23}, which is faster.
This method draws $b > 1$ random pivots with replacement according to the distribution \cref{eq:first_step}.
Then it updates the approximation and residual using the $b$ pivots.
For target rank $k$, the block size $b = k/10$ is often used in applications.
Pseudocode for block \RPCholesky is available in \cref{alg:block_rpc} in the appendix.

This paper resolves several open empirical and theoretical questions about the performance
of block \RPCholesky.
First, new experiments confirm
that simple \RPCholesky and block \RPCholesky provide similar
approximation accuracy on typical problems (\cref{fig:performance}), but there are challenging examples
where the simple method is over $100\times$ more
accurate than the block method (\cref{fig:initial,fig:performance}).
Second, new error bounds guarantee the eventual success of
block \RPCholesky for
any block size $b$ (\cref{thm:main_bound}).
However, the bounds show that the
efficiency of the block method can degrade when the target matrix has rapid
eigenvalue decay.

\subsubsection{Accelerated \RPCholesky} \label{sec:accelerated_intro}

The main contribution of this paper is a new accelerated \RPCholesky algorithm,
which combines the best features of the simple and block methods.
Accelerated \RPCholesky produces the same distribution of random pivots as simple \RPCholesky and enjoys the same accuracy guarantees, while taking advantage of the submatrix access model
to improve the runtime.

At each step, accelerated \RPCholesky proposes a block of $b$ random pivots.
Then the algorithm thins the list of proposals using rejection sampling,
and it updates the matrix approximation using the accepted pivots.
As compared with block \RPCholesky, the new feature of accelerated \RPCholesky
is the rejection sampling step, which simulates the pivot distribution of simple \RPCholesky.

Here is a summary of the accelerated \RPCholesky method;
a comprehensive explanation appears in \cref{sec:accelerated}.
The user fixes a block size parameter $b \geq 1$.
Starting with the approximation $\smash{\Ahat}^{(0)}\coloneqq \mat{0}$ and the residual $\mat{A}^{(0)}\coloneqq \mat{A}$, accelerated \RPCholesky builds up a low-rank approximation $\smash{\Ahat}^{(t)}$ with rank at most $bt$ by means of the following procedure.
For $i=0,\ldots,t-1$,
\begin{enumerate}
    \item \textbf{Propose a block of random pivots.} Propose a block of independent and identically distributed random pivots $\{s_{bi+1}',\ldots,s_{b(i+1)}'\}$ with
    \begin{equation*}
        \prob\{s_\ell' = j\} = \frac{\mat{A}^{(i)}(j,j)}{\tr \mat{A}^{(i)}}.
    \end{equation*}
    \item \textbf{Rejection sampling.}
    Initialize the incremental pivot set $\set{S}_i \gets \emptyset$.
    For each $\ell = bi+1,\ldots,b(i+1)$, 
    accept the pivot $s_{\ell}'$ with probability
    \begin{equation*}
        p(s_\ell') = \frac{\mat{A}^{(i)}(s_\ell' ,s_\ell') - \mat{A}^{(i)}(s_\ell',\set{S}_i)\mat{A}^{(i)}(\set{S}_i,\set{S}_i)^\dagger \mat{A}^{(i)}(\set{S}_i,s_\ell')}{\mat{A}^{(i)}(s_\ell' ,s_\ell')},
    \end{equation*}
    and update $\set{S}_i \gets \set{S}_i \cup \{ s_\ell' \}$. Otherwise, reject the pivot and do nothing.
    \item \textbf{Update.} Update the approximation:
    \begin{align*}
    \smash{\Ahat}^{(i+1)}  &\coloneqq \smash{\Ahat}^{(i)}  + \mat{A}^{(i)}(:,\set{S}_i)\mat{A}^{(i)}(\set{S}_i,\set{S}_i)^\dagger \mat{A}^{(i)}(\set{S}_i,:); \\
    \mat{A}^{(i+1)}  &\coloneqq \mat{A}^{(i)}  - \mat{A}^{(i)}(:,\set{S}_i)\mat{A}^{(i)}(\set{S}_i,\set{S}_i)^\dagger \mat{A}^{(i)}(\set{S}_i,:).
    \end{align*}
    For numerical stability and efficiency, the pseudoinverse $\mat{A}^{(i)}(\set{S}_i,\set{S}_i)^\dagger$ should be computed and applied via a Cholesky decomposition $\mat{A}^{(i)}(\set{S}_i,\set{S}_i)$.
\end{enumerate}
Because of the possibility of rejections, this algorithm uses slightly more entry look-ups than simple \RPCholesky to generate a rank-$k$ approximation.
However, it is much faster in practice due to blockwise computations.
\Cref{alg:accelerated_rpc} provides an optimized implementation.

\subsection{Illustrative comparison}

\begin{figure}[t]
    \centering
    \includegraphics[width=0.48\textwidth]{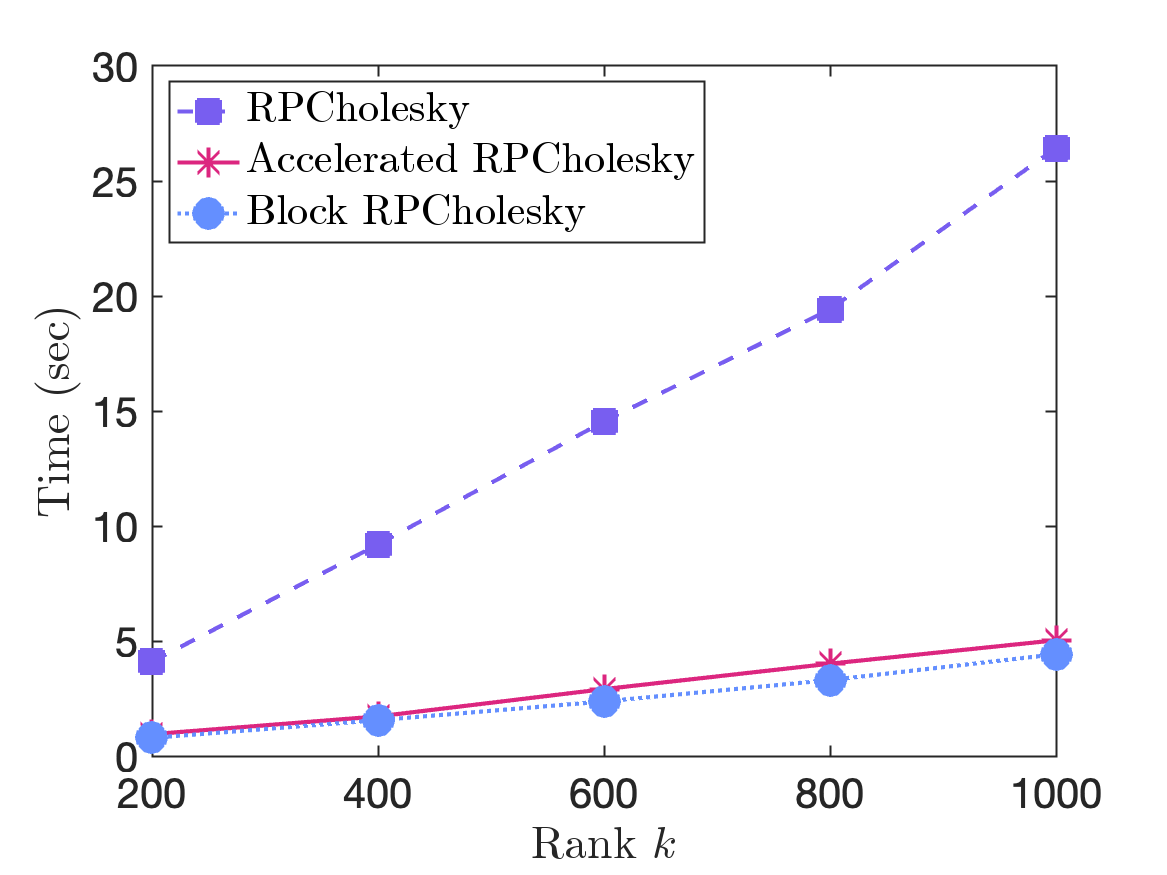}
    \includegraphics[width=0.48\textwidth]{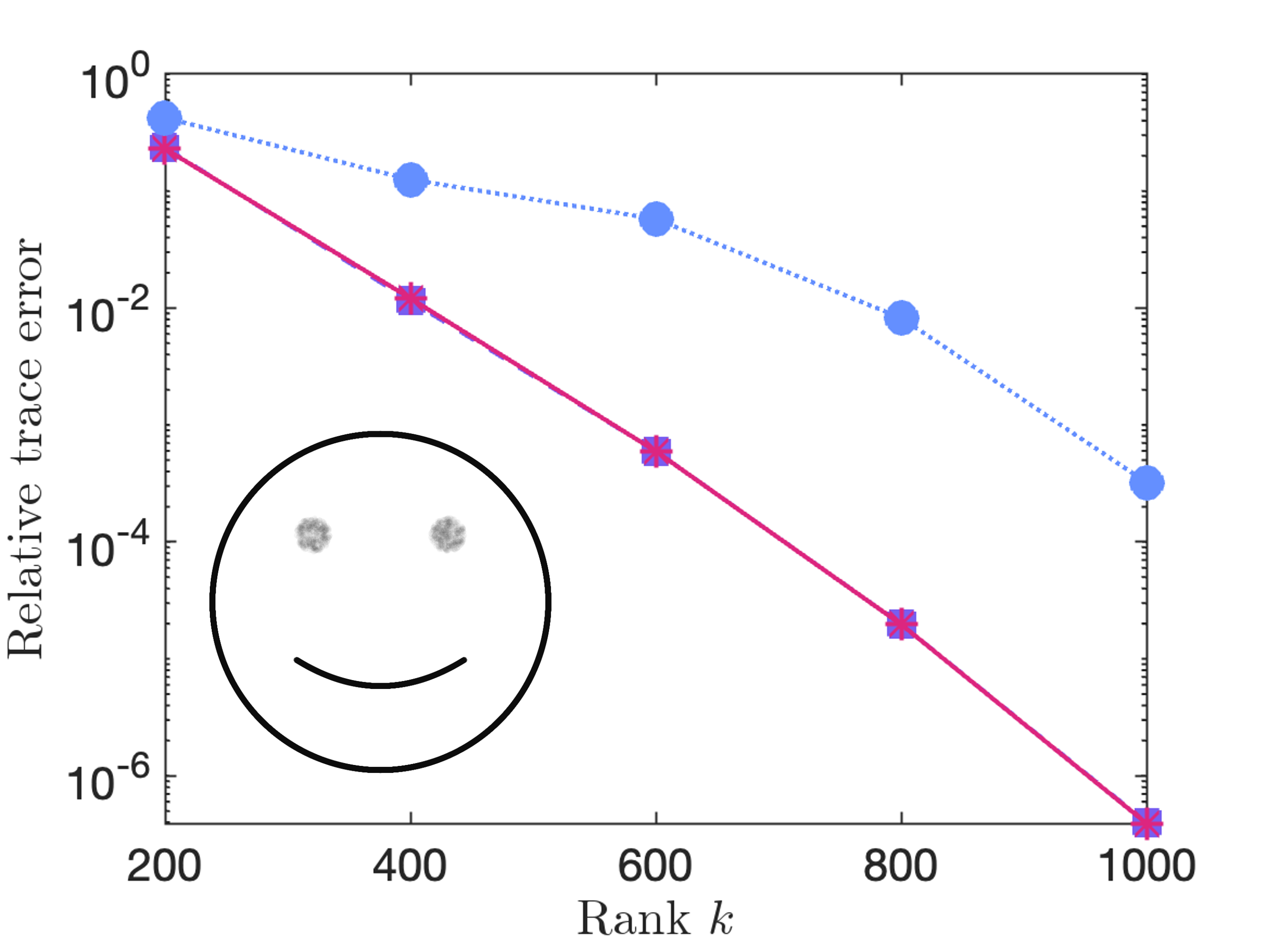}
    
    \caption{Runtime (left) and relative trace norm error \cref{eq:relative-trace-error} (right) of \RPCholesky, block \RPCholesky, and accelerated \RPCholesky applied to a Gaussian kernel matrix $\mat{A}$ associated with $10^5$ points in $\real^2$ forming a smile (inset, right).
    The blocked methods use block size $b=120$.}
    \label{fig:initial}
\end{figure}

\Cref{fig:initial} presents a stylized example where simple \RPCholesky, block \RPCholesky, and accelerated \RPCholesky are applied to a Gaussian kernel matrix, which is generated from $N = 10^5$ data points forming a smile in $\real^2$ of radius $10$.
The bandwidth is set to $\sigma = 0.2$.
This example is only intended to illustrate the differences between the \RPCholesky-type algorithms; \cref{sec:broad} provides more extensive comparisons on a testbed of synthetic data, random point clouds, and real-world data.

As shown in the left panel of \cref{fig:initial}, both block \RPCholesky and accelerated \RPCholesky are $4\times$ to $5\times$ faster than simple \RPCholesky on this example.
The right panel of \cref{fig:initial} shows the relative trace error 
\begin{equation} \label{eq:relative-trace-error}
    \frac{\tr(\mat{A}-\Ahat)}{\tr(\mat{A})}
\end{equation}
of the methods.
Accelerated \RPCholesky and simple \RPCholesky give the same accuracy, so the error curves coincide with each other.
Block \RPCholesky fares much worse, attaining $800\times$ higher error than the other methods at rank $k=1000$.
This example supports the basic takeaway message: \textbf{accelerated \RPCholesky produces the same distribution of random outputs as simple \RPCholesky while being much faster.}

\subsection{Outline}

This paper explores the theoretical and empirical properties of accelerated \RPCholesky and compares it to simple and block \RPCholesky.
\Cref{sec:accelerated} introduces accelerated \RPCholesky and discusses efficient implementations.
\Cref{sec:numerics} presents numerical experiments.
\Cref{sec:theory} analyzes accelerated \RPCholesky and block \RPCholesky.
\Cref{sec:qr} extends these methods and proposes an accelerated randomly pivoted \QR method.

\subsection{Notation} \label{sec:notation}
Entries and submatrices of $\mat{A}$ are expressed using MATLAB notation.
For example, $\mat{A}(:,i)$ represents the $i$th column of $\mat{A}$ and $\mat{A}(\set{S},:)$ denotes the submatrix of $\mat{A}$ with rows indexed by the set $\set{S}$.
The conjugate transpose of a matrix $\mat{F}$ is denoted as $\mat{F}^*$, and the Moore--Penrose pseudoinverse is $\mat{F}^{\dagger}$.
The inverse of the conjugate transpose of a nonsingular matrix is denoted $\mat{F}^{-*} \coloneqq (\mat{F}^{-1})^* = (\mat{F}^*)^{-1}$.
The matrix $\mat{\Pi}_{\mat{F}}$ is the orthogonal projector onto the column span of $\mat{F}$.
The function $\lambda_i(\mat{A})$ outputs the $i$th largest eigenvalue of a psd matrix $\mat{A}$.
The symbol $\preceq$ denotes the psd order on Hermitian matrices: $\mat{H} \preceq \mat{A}$ if and only if $\mat{A} - \mat{H}$ is psd.
A matrix $\mat{A}$ is described as ``rank-$r$'' when $\rank \mat{A} \leq r$.  
Last, $\norm{\cdot}$ denotes the vector $\ell_2$ norm.

\section{Faster \RPCholesky by rejection sampling} \label{sec:accelerated}

This section explains how \RPCholesky can be implemented efficiently with rejection sampling (\cref{sec:rejection,sec:rejection2}).
Afterward,
\cref{sec:accelerated_alg} introduces and provides pseudocode for accelerated \RPCholesky.
\Cref{sec:low-memory} explains how to implement accelerated \RPCholesky with more operations but lower memory requirements.

\subsection{Implementing \RPCholesky with rejection sampling} \label{sec:rejection}
Rejection sampling is a standard approach for sampling from a challenging probability distribution \cite[Sec.~2.2]{liu2004monte}.
This section shows how to use rejection sampling to sample the pivots $s_1,\ldots,s_k$ according to the \RPCholesky distribution.
Rejection sampling was originally applied to \RPCholesky in \cite{EM23}, building on related ideas from \cite{RO19}.

Consider the following setting.
Suppose that one has already sampled pivots $\set{S}_i = \{s_1,\ldots,s_i\}$, which define the rank-$i$ approximation $\Ahat^{(i)} = \mat{A}(:,\set{S}_i) \mat{A}(\set{S}_i,\set{S}_i)^\dagger \mat{A}(\set{S}_i,:)$ and the residual $\mat{A}^{(i)} = \mat{A} - \Ahat^{(i)}$.
In addition, suppose that one has access to an upper bound $\vec{u}$ on the diagonal of the residual $\mat{A}^{(i)}$:
\begin{equation*}
    \vec{u}(j) \ge \mat{A}^{(i)}(j,j) \quad \text{for each } j = 1, \ldots, N.
\end{equation*}
The diagonal of the residual decreases at each step of the \RPCholesky procedure: $\mat{A}^{(i)}(j,j) \le \mat{A}^{(i-1)}(j,j) \le \cdots \le \mat{A}(j,j)$ for each index $j$.  Thus, $\vec{u} = \diag(\mat{A})$ is one possible choice for the upper bound.

To select the next pivot $s_{i+1}$, rejection sampling executes the following loop:
\boxedtext{\textbf{Rejection sampling loop:}
\begin{enumerate}
    \item \textbf{Propose.} Draw a random pivot $s$ according to the \emph{proposal distribution} $\vec{u}$, given by $\prob \{ s = j \} =\vec{u}(j) / \sum_{k=1}^n \vec{u}(k)$.
    \item \textbf{Accept/reject.}
    With probability $\mat{A}^{(i)}(s,s) / \vec{u}(s)$, accept the pivot and set $s_{i+1} \gets s$.
    Otherwise, reject the pivot and return to step 1.
\end{enumerate}
}
\noindent Notice that steps 1 and 2 are repeated until a pivot $s_{i+1}$ is accepted.
Thus, the procedure generates a set $\set{S}_{i+1} \gets \set{S}_i \cup \{s_{i+1}\}$ of $i+1$ pivots, defining a rank-$(i+1)$ approximation $\Ahat^{(i+1)}$.
Continuing in this way, one can sample pivots $s_{i+2},s_{i+3},\ldots$ and construct a low-rank approximation of the desired rank $k$.

The advantage of rejection sampling over the standard \RPCholesky implementation (\cref{alg:rpcholesky}) is that rejection sampling only evaluates the diagonal entries $\diag \mat{A}^{(i)}(s, s)$ at the proposed pivot indices $s$,
whereas \cref{alg:rpcholesky} evaluates all the entries of $\diag \mat{A}^{(i)}$ at every step.
Nonetheless, the efficiency of rejection sampling depends on the typical size of the acceptance probability $\mat{A}^{(i)}(s,s)/\vec{u}(s)$.
A small acceptance probability leads to many executions of the loop 1--2 and thus a slow algorithm.
Accelerated \RPCholesky avoids this issue by periodically updating $\vec{u}$ to keep the acceptance rate from becoming too low.

\subsection{Efficient rejection sampling in the submatrix access model} \label{sec:rejection2}
This section describes how to implement \RPCholesky with rejection sampling
in the submatrix access model.
The resulting procedure, called \RejectChol, serves as a subroutine in the full accelerated \RPCholesky algorithm,
which is developed in \cref{sec:accelerated_alg}.

\RejectChol works as follows.
Assume the proposal distribution $\vec{u} \coloneqq \diag \mat{A}$,
and assume the initial pivot set is empty: $\set{S}_0 = \emptyset$.
To take advantage of the submatrix access model, begin by sampling a block
$\set{S}' = \{ s_1', \dots, s_b'\}$ of $b \geq 1$  proposals
independently from the distribution
\begin{equation*}
    \mathbb{P}\{s' = j\} = \frac{\mat{A}(j,j)}{\tr(\mat{A})} \quad \text{for } j=1,\ldots,N.
\end{equation*}
For efficiency,
generate the entire submatrix $\mat{A}(\set{S}', \set{S}')$.
Finally, perform $b$ 
pivot accept or reject steps with probabilities given by the diagonal submatrix elements
to obtain a set $\set{S} \subseteq \set{S}'$ of accepted pivots.
Each time a pivot is accepted, the procedure performs a Cholesky step to
eliminate the accepted column from the submatrix and to 
update the acceptance probabilities.
Pseudocode for \RejectChol is given in \cref{alg:rejection_rpc}.
\begin{algorithm}[t]
    \caption{\RejectChol subroutine} \label{alg:rejection_rpc}
    \begin{algorithmic}
    \Require Psd matrix $\mat{H} \in \mathbb{C}^{b\times b}$ and proposed pivots $\set{S}' = \{ s_1',\ldots,s_b' \}$
    \Ensure Accepted pivots $\set{S} \subseteq \set{S}'$; Cholesky factor $\mat{L}$
    \State $\mat{L} \gets \mat{0}_{b\times b}$, $\set{T} \gets \emptyset$ 
    \State $\vec{u} \gets \diag (\mat{H})$
    \For{$i = 1,\ldots,b$}
        \If{$u_i \cdot \Call{Rand}{\,} < h_{ii}$} \Comment{Accept or reject}
            \State $\set{T} \gets \set{T} \cup \{i\}$ \Comment{If accept, induct pivot}
            \State $\mat{L}(i:b,i) \gets \mat{H}(i:b,i)/\sqrt{\mat{H}(i,i)}$ \Comment{Compute Cholesky factor}
            \State $\mat{H}(i+1:b,i+1:b) \gets \mat{H}(i+1:b,i+1:b) - \mat{L}(i+1:b,i) \mat{L}(i+1:b,i) ^*$ 
        \EndIf
    \EndFor
    \State $\mat{L} \gets \mat{L}(\set{T}, \set{T})$ \Comment{Delete rejected rows and columns}
    \State $\set{S} \gets \{ s_i' : i \in \set{T} \}$ \Comment{Get accepted pivots}
    \end{algorithmic}
\end{algorithm}
As inputs, it takes the proposed pivots $\set{S}'$ and the submatrix $\mat{H} \coloneqq \mat{A}(\set{S}', \set{S}')$, and it returns the set $\set{S} \subseteq \set{S}'$ of accepted pivots and a Cholesky factorization $\mat{A}(\set{S},\set{S}) = \mat{L}\mat{L}^*$.

In contrast to the procedure in the previous section, \RejectChol only makes $b$ proposals in total,
and the set $\set{S} \subseteq \set{S}'$ of accepted pivots has random cardinality $1 \leq |\set{S}| \leq b$.
The first pivot $s_1'$ is always accepted because the proposal distribution $\vec{u} = \diag(\mat{A})$.
Conditional on the number of acceptances $|\set{S}|$, the accepted pivots $s_1,\ldots,s_{|\set{S}|}$ have the same distribution as $|\set{S}|$ pivots selected by the standard \RPCholesky algorithm (\cref{alg:rpcholesky}) applied to the full matrix $\mat{A}$.

\subsection{Accelerated \RPCholesky} \label{sec:accelerated_alg}
The shortcoming of the block rejection sampling approach in the previous section is that the acceptance rate declines as the diagonal of the residual matrix decreases. 
Accelerated \RPCholesky remedies this weakness by employing a multi-round sampling procedure.  At the beginning of each round, the sampling distribution $\vec{u} \gets \diag(\mat{A} - \Ahat)$ is updated using the current low-rank approximation $\Ahat$.

\begin{algorithm}[t]
    \caption{Accelerated \RPCholesky: Standard implementation} \label{alg:accelerated_rpc}
    \begin{algorithmic}
        \Require Psd matrix $\mat{A} \in \mathbb{C}^{N\times N}$; block size $b$; number of rounds $t$
        \Ensure Matrix $\mat{F} \in \mathbb{C}^{N \times k}$ defining low-rank approximation $\Ahat = \mat{F}\mat{F}^*$; pivot set $\set{S}$ 
        \State Initialize $\mat{F} \leftarrow \mat{0}_{N \times 0}$,
        $\set{S} \leftarrow \emptyset$, and
        $\vec{u} \leftarrow \diag \mat{A}$
        \For{$i = 0$ to $t - 1$}
            \State \underline{\textit{Step 1: Propose a block of pivots}} \vspace{0.5em}
            \State Sample $s_{ib+1}', \ldots, s_{(i+1)b}' \stackrel{\rm iid}{\sim} \vec{u}$ and set $\set{S}_i' \gets \{s_{ib+1}', \ldots, s_{(i+1)b}'\}$
            \State 
            \State \underline{\textit{Step 2: Rejection sampling}} \vspace{0.5em}
            \State $\mat{H} \gets \mat{A}(\set{S}_i', \set{S}_i') - \mat{F}(\set{S}_i',:) \mat{F}(\set{S}_i',:)^*$
            \State $\set{S}_i ,\mat{L}\gets \Call{\RejectChol}{\set{S}_i',\mat{H}}$
            \State $\set{S} \gets \set{S} \cup \set{S}_i$ \Comment{Update pivots} 
            \State 
            \State \underline{\textit{Step 3: Update low-rank approximation and proposal distribution}} \vspace{0.5em}
            \State $\mat{G} \gets (\mat{A}(:,\set{S}_i) - \mat{F} \mat{F}(\set{S}_i,:)^*) \mat{L}^{-*}$
            \State $\mat{F}\gets \onebytwo{\mat{F}}{\mat{G}}$
            \State $\vec{u} \gets \vec{u} - \Call{SquaredRowNorms}{\mat{G}}$ \Comment{Update diagonal}
            \State $\vec{u} \gets \max \{\vec{u},\vec{0}\}$ \Comment{Helpful in floating point arithmetic}
        \EndFor
    \end{algorithmic}
\end{algorithm}

In detail, accelerated \RPCholesky works as follows.
Fix a number of rounds $t\ge 1$ and block size $b\ge 1$, and initialize $\Ahat \gets \mat{0}$ and $\vec{u} \gets \diag(\mat{A})$.
For $i = 0,1,\ldots,t-1$, perform the following three steps:
\begin{enumerate}
    \item \textbf{Propose a block of pivots.} Sample $b$ pivots independently with replacement according to the distribution $\prob \{ s = j \} = \vec{u}(j) / \sum_{k=1}^n \vec{u}(k)$. 
    \item \textbf{Rejection sampling.} Use \RejectChol (\cref{alg:rejection_rpc}) to downsample the proposed pivots.
    \item \textbf{Update.} Update the low-rank approximation $\Ahat$ and the diagonal vector $\vec{u}\gets \diag(\mat{A} - \Ahat)$ to reflect the newly accepted pivots.
\end{enumerate}
Optimized pseudocode for accelerated \RPCholesky is given in \cref{alg:accelerated_rpc}.
As output, \cref{alg:accelerated_rpc} returns an approximation $\Ahat \approx \mat{A}$ of rank $t\le k \leq b t$, where $k$ is the total number of accepted pivots.
The rank $k$ is at least $t$ since the first pivot of each round is always accepted.
For convenience and computational efficiency, $\Ahat$ is reported in factored form $\Ahat = \mat{F}\mat{F}^*$.

The computational timing of \cref{alg:accelerated_rpc} can be summarized as follows.
Step 1 makes a negligible contribution to the overall timing, requiring $\order(N + b)$ operations.
Step 2 is also relatively fast, 
since it requires just $\order((b+k) b^2)$ operations, independent of the dimension $N$ of the matrix.
Step 3 is typically the slowest, as it generates and processes a large block of $|\set{S}_i| \leq b$ columns in $\mathcal{O}(k |\set{S}_i| N)$ operations.
However, the user should take care that the acceptance rate does not fall below the level $\order(b / N)$ because then step 2 also becomes a computational bottleneck.

\subsection{Low-memory implementation} \label{sec:low-memory}

As a limitation, \Cref{alg:accelerated_rpc} requires storing an $N \times k$ factor matrix, which is prohibitively expensive when the approximation rank and the matrix dimensions are very large, say, $k \geq 10^4$ and $N \geq 10^6$.
For a problem of this particular size, storing the factor matrix requires $\geq 80{\rm GB}$ of memory in double-precision floating-point format, which exceeds the available memory on many machines.

\begin{algorithm}[t]
\caption{Accelerated \RPCholesky: Low-memory implementation} \label{alg:accelerated_rpc_low_memory}
\begin{algorithmic}
    \Require Psd matrix $\mat{A} \in \mathbb{C}^{N\times N}$; block size $b$; number of rounds $t$
    \Ensure Pivot set $\set{S}$ defining low-rank approximation $\Ahat = \mat{A}(:,\set{S})\mat{A}(\set{S},\set{S})^\dagger \mat{A}(\set{S},:)$; Cholesky factor $\mat{L}$ with $\mat{A}(\set{S}, \set{S}) = \mat{L} \mat{L}^*$
    \State Initialize $k\gets t\cdot b$, $\mat{L} \leftarrow \mat{0}_{0 \times 0}$,
    $\set{S} \leftarrow \emptyset$, and
    $\vec{u} \leftarrow \diag \mat{A}$
    \For{$i = 0$ to $t - 1$}
        \State \underline{\textit{Step 1: Propose a block of pivots}} \vspace{0.5em}
        \State Sample $s_{ib+1}', \ldots, s_{(i+1)b}' \stackrel{\rm iid}{\sim} \vec{u}$ and set $\set{S}_i' \gets \{s_{ib+1}', \ldots, s_{(i+1)b}'\}$
        \State
        \State \underline{\textit{Step 2: Rejection sampling}} \vspace{0.5em}
        \State $\mat{H} \gets \mat{A}(\set{S}_i', \set{S}_i') - \mat{A}(\set{S}_i',\set{S})\mat{L}^{-*} \mat{L}^{-1} \mat{A}(\set{S}_i',\set{S})$
        \State $\set{S}_i, \mat{L}_i\gets \Call{\RejectChol}{\mat{H},\set{S}_i'}$
        \State $\set{S} \gets \set{S} \cup \set{S}_i$ \Comment{Update pivots}
        \State 
        \State \underline{\textit{Step 3: Update low-rank approximation and proposal distribution}} \vspace{0.5em}
        \For{$j = 0$ to $t-1$} \Comment{Update diagonal, $k$ entries at a time}
            \State $\set{R} \gets \{jb+1,\ldots,(j+1)b\}$ \Comment{Chunk of $b$ rows}
            \State Evaluate the submatrix $\mat{A}(\set{R},\set{S})$
            \State $\mat{G} \gets (\mat{A}(\set{R},\set{S}_i) - \mat{A}(\set{R},\set{S}\setminus\set{S}_i) \mat{L}^{-*} \mat{L}^{-1} \mat{A}(\set{S}\setminus\set{S}_i,\set{S}_i)) \mat{L}_i^{-*} $
            \State $\vec{u}(\set{R}) \gets \vec{u}(\set{R}) - \Call{SquaredRowNorms}{\mat{G}}$
        \EndFor
        \State $\mat{L} \gets \twobytwo{\mat{L}}{\mat{0}}{\mat{A}(\set{S}_i,\set{S}\setminus\set{S}_i)\mat{L}^{-*}}{\mat{L}_i}$ 
        \Comment{Update Cholesky factor}
        \State $\vec{u} \gets \max \{\vec{u},\vec{0}\}$ \Comment{Ensure $\vec{u}$ remains nonnegative} 
    \EndFor
\end{algorithmic}
\end{algorithm}

\Cref{alg:accelerated_rpc_low_memory} addresses the memory bottleneck with an alternative implementation of accelerated \RPCholesky that only requires $\mathcal{O}(N + k^2)$ storage and repeatedly regenerates entries of $\mat{A}$ on an as-needed basis.
The method produces the same distribution of pivots and acceptance probabilities as \cref{alg:accelerated_rpc}, and the computational cost is still $\mathcal{O}(N k^2)$ operations.
However, the number of kernel matrix evaluations increases from $\mathcal{O}(N k)$ to $\mathcal{O}(N k^2)$ because the kernel matrix entries are regenerated when they are needed.

The low-memory implementation in \cref{alg:accelerated_rpc_low_memory} does not explicitly store an $N \times k$ factor matrix.
Rather, it maintains a Cholesky factorization $\mat{A}(\set{S},\set{S}) = \mat{L} \mat{L}^* \in \mathbb{C}^{k \times k}$ that implicitly defines the low-rank approximation:
\begin{equation}
    \Ahat = \mat{A}(:,\set{S})\mat{A}(\set{S},\set{S})^{-1} \mat{A}(\set{S},:) = \mat{A}(:,\set{S}) \mat{L}^{-*}\mat{L}^{-1} \mat{A}(\set{S},:).
\end{equation}
The submatrix $\mat{A}(:,\set{S})$ is not stored, and it is regenerated as needed.
The implicit low-rank approximation $\mat{\hat{A}}$ can be accessed via matrix--vector products, and it can be employed, for example, to speed up kernel machine learning algorithms \cite{DEF+23}.

\subsection{Related work: Robust blockwise random pivoting}

In recent work \cite{DCMP23}, Dong, Chen, Martinsson, and Pearce introduced \emph{robust blockwise random pivoting}, a variant of block randomly pivoted \QR for rectangular matrix low-rank approximation; this method can also be extended to provide a partial pivoted Cholesky decomposition for psd low-rank approximation.
Similar to accelerated \RPCholesky, RBRP works by first sampling a block of random pivots and downsampling them to a smaller set of representatives.
A comparison of our method with their approach appears in \cref{sec:rbrp}.

\section{Numerical results} \label{sec:numerics}

This section presents numerical experiments highlighting the speed and accuracy of accelerated \RPCholesky, both for a testbed of 125 kernel matrices (\cref{sec:broad})
and for 8 potential energy calculations from molecular chemistry (\cref{sec:case}).
The code to replicate the experiments can be found in the \texttt{block\_experiments} folder of \url{https://github.com/eepperly/Randomly-Pivoted-Cholesky}.
In all these experiments, a target rank $k$ is set and block \RPCholesky and accelerated \RPCholesky are run for as many rounds $t$ as needed to select $k$ pivots.
The full-memory implementation of accelerated \RPCholesky (\cref{alg:accelerated_rpc}) is used in all experiments.

\subsection{Testbed of 125 kernel matrices} \label{sec:broad}

The first round of experiments evaluates the accuracy and speed of accelerated \RPCholesky using over 100 kernel matrices, although some matrices were constructed from the same data with different kernel functions $\kappa$.
The target matrices were generated from the following synthetic, random, and real-world data sets:
\begin{itemize}
    \item \textbf{Challenging synthetic examples.} Three synthetic data were constructed to provide challenging examples of outliers and imbalanced classes of data points.
    The first two data sets depict a smile and a spiral in $\real^2$, scaled to three different length scales.
    The third data set represents a Gaussian point cloud in $\real^{20}$, with 50 or 500 or 5000 points perturbed to become outliers.
    These synthetic data sets include $N=10^5$ points, and the kernel function is Gaussian with bandwidth $\sigma = 1$.
    \item \textbf{Random point clouds.}
    Four random data sets were simulated, containing $N=10^5$ standard Gaussian data points in dimension $d \in \{ 2,10,100,1000\}$.
    The kernel function is Gaussian, Mat\'ern-$3/2$, or $\ell_1$-Laplace with bandwidth $\sigma \in \{ \sqrt{d}/2, \sqrt{d}, 2 \sqrt{d} \}$.
    \item \textbf{Real-world data.} 28 real-world data sets were downloaded from LIBSVM, OpenML, and UCI.
    They include the 20 examples from \cite[Tab.~1]{DEF+23}, supplemented with \texttt{a9a}, \texttt{cadata}, \texttt{Click\_prediction\_small}, \texttt{phishing}, \texttt{MiniBoonNE}, \texttt{santander}, \texttt{skin\_nonskin}, and \texttt{SUSY}.
    These data sets range in size from $N=4.3\times 10^4$ to $N = 1.3\times 10^6$ points, but the larger data sets were randomly subsampled to a maximum size of $N = 10^5$ points.
    The kernel function is either 
    Gaussian, Mat\'ern-$3/2$, or $\ell_1$-Laplace.
    The bandwidth is the median pairwise distance from a random subsample of $10^3$ points.
\end{itemize}
For each kernel matrix, the approximation rank is $k = 1000$ and the block size is $b=150$.
A small number of examples were discarded which possess a rank-$k$ approximation of relative trace error $<10^{-13}$; after this filtering process, there were 125 examples.

\begin{figure}[t]
    \centering
    \includegraphics[width=0.9\textwidth]{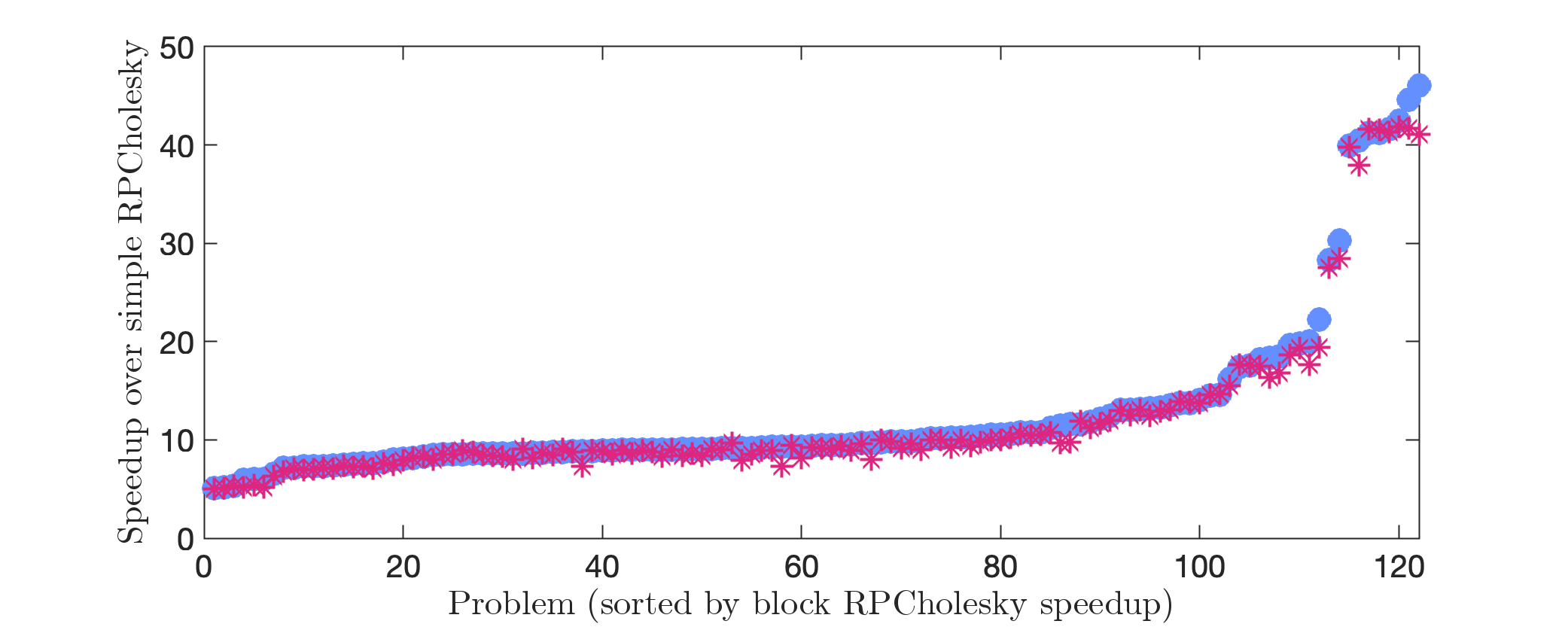}

    \includegraphics[width=0.9\textwidth]{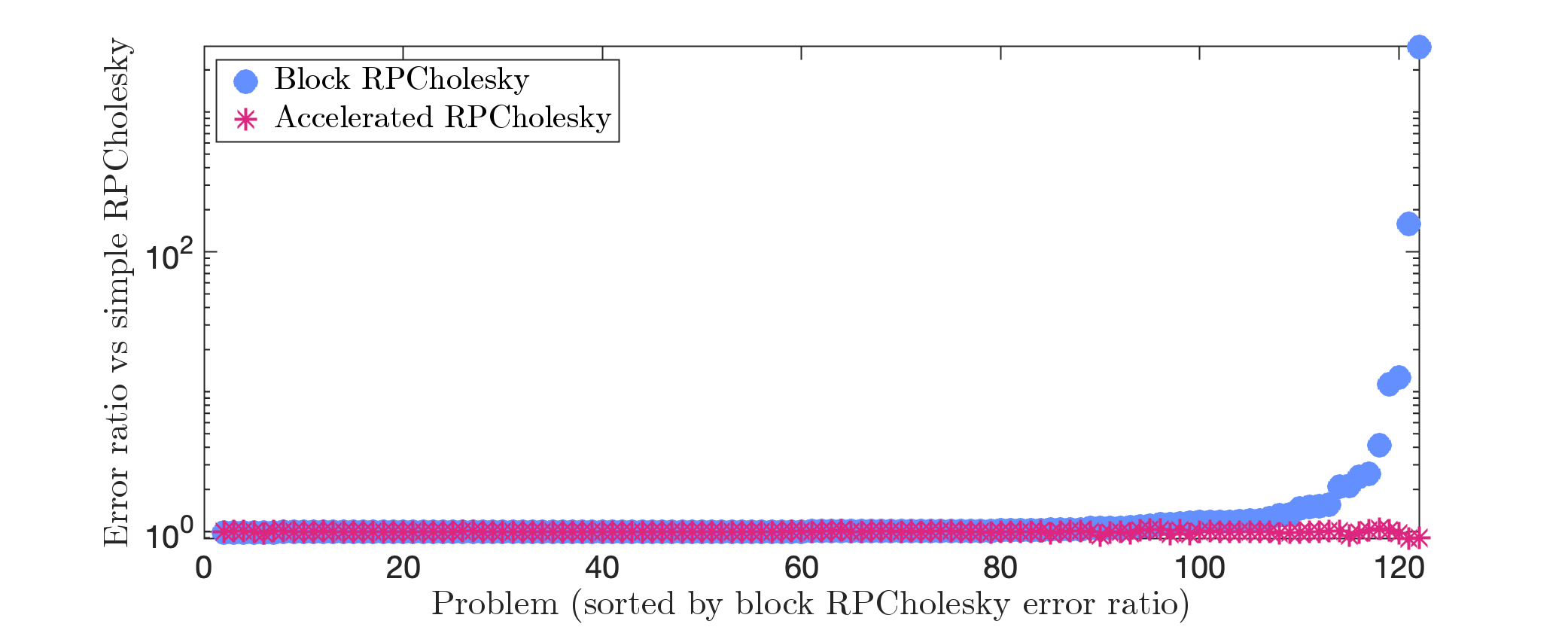}
    
    \caption{Speed-up factor (top) and trace error ratio (eq.~\cref{eq:ratio}, bottom) of block and accelerated \RPCholesky compared to simple \RPCholesky on the testbed of 125 kernel matrices.
    Examples are sorted by the block \RPCholesky speed-up factor (top) or the block \RPCholesky error ratio (bottom).}
    \label{fig:performance}
\end{figure}

Results are shown in \cref{fig:performance}.
The top panel shows the speed-up factor of block and accelerated \RPCholesky over simple \RPCholesky.
The speed-up factor is at least $5\times$ for all kernel matrices, and it reaches a level of $40\times$ to $50\times$ for $5\%$ of the matrices.
Consistent with \cref{fig:generation}, the largest speedups are for high-dimensional data sets (real or synthetic) with the Gaussian or Mat\'ern kernels.
On most examples, the acceptance rate for accelerated \RPCholesky remains at or above 50\%.
However, for a few problems, like \texttt{COMET\_MC\_SAMPLE} with a Gaussian kernel, the acceptance rate for accelerated \RPCholesky drops as low as 3\%.
As might be expected, the examples where accelerated \RPCholesky has a high rejection rate coincide with the problems where block \RPCholesky is much less accurate than accelerated \RPCholesky, demonstrating the importance of filtering the pivot set for these examples.

The bottom panel shows the ratio of trace norm errors
\begin{equation}
\label{eq:ratio}
    \frac{\tr(\mat{A} - \Ahat)}{\tr(\vphantom{{\mat{\hat{A}}}}\mat{A} - \Ahat_{\rm RPC})},
\end{equation}
where $\Ahat$ is computed by block or accelerated \RPCholesky and $\Ahat_{\rm RPC}$ is produced by simple \RPCholesky.
As expected, accelerated \RPCholesky and \RPCholesky have the same approximation quality.
For 100 out of 125 kernel matrices, block \RPCholesky leads to $1.00\times$ to $1.16\times$ the simple \RPCholesky error.
For the remaining matrices, block \RPCholesky produces a significantly worse approximation than simple \RPCholesky; the error is $3000\times$ worse for the most extreme problem.

\subsection{Potential energy surfaces in molecular chemistry}
\label{sec:case}

Molecular dynamics simulations rely on a fast-to-compute model for the potential energy of an atomic system.
Recent work has explored fitting potential energies with machine learning methods such as kernel ridge regression (KRR) \cite{chmiela2017machine,chmiela2018towards,chmiela2023accurate,unke2021machine}.
However, despite the ubiquity of KRR in molecular chemistry, there is a surprising lack of consensus over the best numerical method for KRR optimization.

Responding to this uncertainty, the present paper demonstrates how \emph{low-rank preconditioners} can accelerate KRR for potential energy modeling and it shows that accelerated \RPCholesky is the best available low-rank approximation method for this use case.
These conclusions are supported by experiments on a benchmark dataset of eight molecules: aspirin, benzene, ethanol, malonaldehyde, naphthalene, salicylic acid, toluence, and uracil.

The experiments in this section use density functional theory data released by Chmiela et al.\ \cite{chmiela2017machine}, available at \href{https://www.sgdml.org/}{sgdml.org}.
This data was prepared as follows:
\begin{itemize}
\item \emph{Data size:} The measurements for each molecule were randomly subsampled to $N = 10^5$ training points and $N_{\rm test} = 10^4$ testing points.
\item \emph{Input data features:} The input data consists of positions $\vec{r}_i\in\real^3$ for each atom $i = 1, \ldots, n_{\rm atoms}$.
Since the potential energy is invariant under rigid motions, the data is transformed into a vector $\vec{x}$ listing the inverse pairwise distances $1 / \lVert \vec{r}_i - \vec{r}_j \rVert$ for $1 \leq i < j \leq n_{\rm atoms}$ \cite{chmiela2017machine}.
Following the usual procedure, these features $\vec{x}$ were standardized by subtracting the mean and dividing by the standard deviation.
The dimensionality of the input data is thus $d = {n_{\rm atoms} \choose 2} \in [36, 210]$, depending on the molecule.
\item \emph{Output data features:} The output data is potential energy measurements in units of kcal mol$^{-1}$. 
The potential energies were centered by subtracting the mean but they were \emph{not} divided by the standard deviation.
The standard deviations are thus $2.3$--$6.0$ kcal mol$^{-1}$, depending on the molecule.
\end{itemize}

This paper employs a standard KRR model \cite{kanagawa2018gaussianprocesseskernelmethods}, in which the potential energy for each molecule is approximated by a linear combination of kernel functions,
\begin{equation}
\label{eq:predictor}
    f(\vec{x}) = \sum\nolimits_{i=1}^N \beta_i \, \kappa(\vec{x}_i, \vec{x}),
\end{equation}
that are centered on the training data points $(\vec{x}_i)_{1 \leq i \leq N}$.
The coefficients $\beta_i \in \mathbb{R}$ are chosen as the solution to the linear system:
\begin{equation}
\label{eq:reduce}
    (\mat{A} + \mu \mathbf{I}) \vec{\beta} = \vec{y}.
\end{equation}
Here, $\vec{y} \in \real^n$ is the vector of potential energy output values, $\mat{A} = (\kappa(\vec{x}_i, \vec{x}_j))_{1 \leq i,j \leq N}$ is the kernel matrix, and $\mu > 0$ is the regularization parameter.
Following \cite{chmiela2017machine}, the present study chooses $\kappa$ to be a Mat{\'e}rn-5/2 kernel
\begin{equation*}
    \kappa(\vec{x}, \vec{x}') = \Bigl(1 + \sqrt{5} \cdot \tfrac{\lVert \vec{x} - \vec{x}' \rVert}{\sigma} + \tfrac{5}{3} \cdot \tfrac{\lVert \vec{x} - \vec{x}' \rVert^2}{\sigma^2}\Bigr)
    \exp\Bigl(-\sqrt{5} \cdot \tfrac{\lVert \vec{x} - \vec{x}' \rVert}{\sigma} \Bigr)
\end{equation*}
with bandwidth $\sigma = \sqrt{d}$.
The regularization parameter is set to $\mu = 10^{-9} N$.

For large $N$, the standard approach for solving \cref{eq:reduce} in both the scientific \cite{chmiela2017machine,unke2021machine,chmiela2018towards,chmiela2023accurate} and mathematical \cite{cutajar2016preconditioning,DEF+23,avron2017faster,alaoui2015fast,gardner2018gpytorch,wang2019exact,frangella2023randomized} literatures is preconditioned conjugate gradient (PCG) \cite[Alg.~11.5.1]{GV13}.
PCG is an iterative algorithm, and each iteration requires a matrix--vector product with $\mat{A} + \mu \Id$ and a linear solve with a preconditioner matrix $\mat{P}$.
To control the number of iterations and the total runtime, it is essential to find an easy-to-invert preconditioner satisfying $\mat{P} \approx \mat{A} + \mu \Id$.
Recent work has promoted the use of
\emph{low-rank preconditioners}, which take the form
\begin{equation}
\label{eq:low_rank}
    \mat{P} = \mat{F} \mat{F}^* + \mu \mathbf{I} \quad \text{for } \mat{F} \in \mathbb{R}^{N \times k},
\end{equation}
where $\mat{F} \mat{F}^* \approx \mat{A}$ is a rank-$k$ approximation.
The community has vigorously debated the best algorithm to compute $\mat{F}$ in practice and in theory \cite{cutajar2016preconditioning,DEF+23,avron2017faster,alaoui2015fast,MM17,rudi2018fast,gardner2018gpytorch,wang2019exact,frangella2023randomized}.

\begin{figure}[t]
    \centering
    \includegraphics[width=.8\textwidth]{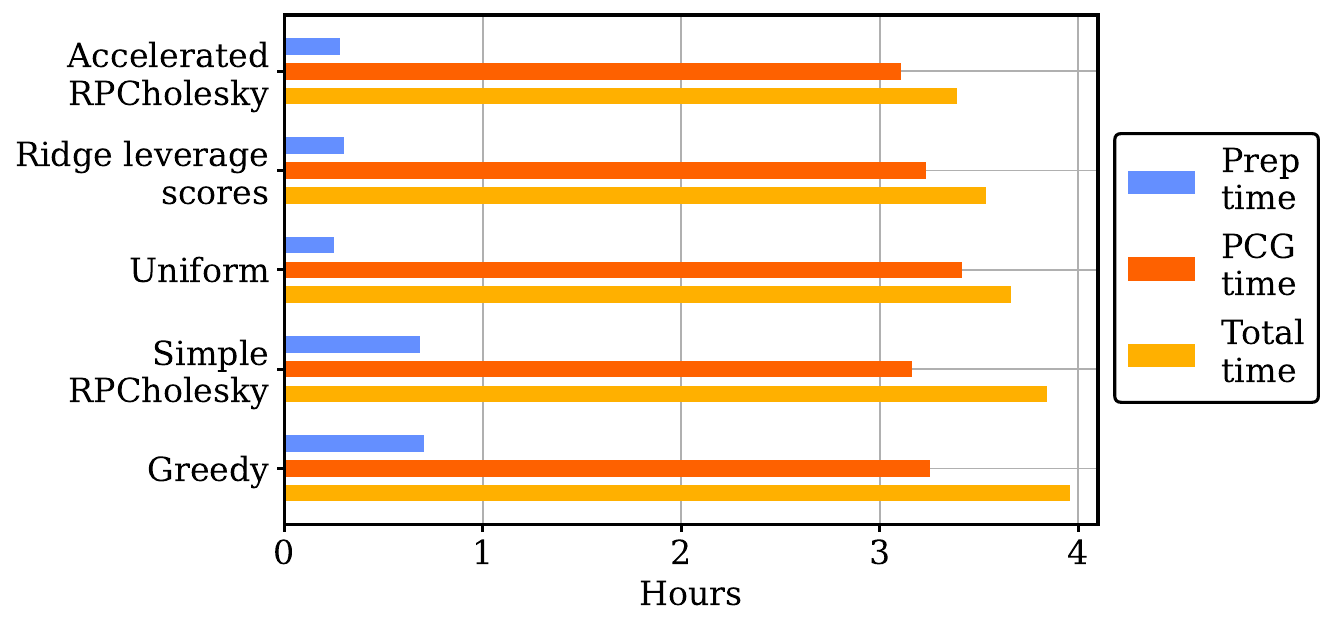}
    
    \caption{Runtime needed to generate the preconditioner (prep time) and run the necessary number of CG iterations until $\lVert \mat{A} \vec{\beta}^{(t)} - \vec{y} \rVert / \lVert \vec{y} \rVert < 10^{-3}$ (PCG time).
    Each calculation is performed once and the results are averaged over the eight molecules.}
    \label{fig:pes}
\end{figure}

\begin{table}[t]
    \centering
    \begin{tabular}{p{.18\textwidth}p{.15\textwidth}p{.1\textwidth}p{.09\textwidth}p{.1\textwidth}p{.18\textwidth}}
    \toprule
    Molecule & Accelerated \newline RPC & Simple \newline RPC & Greedy & Uniform & Ridge leverage scores \\\midrule
    Aspirin & 0.091 & 0.090 & \textbf{0.085} & 0.097 & 0.099 \\
    Benzene & \textbf{0.064} & \textbf{0.064} & 0.069 & 0.067 & 0.070 \\
    Ethanol & \textbf{0.078} & \textbf{0.078} & 0.087 & 0.079 & 0.080 \\
    Malonaldehyde & \textbf{0.073} & \textbf{0.073} & 0.084 & 0.074 & 0.075 \\
    Naphthalene & \textbf{0.106} & \textbf{0.106} & 0.108 & 0.109 & 0.111 \\
    Salicylic acid & \textbf{0.093} & \textbf{0.093} & 0.095 & 0.096 & 0.098 \\
    Toluene & 0.104 & \textbf{0.103} & 0.109 & 0.106 & 0.107 \\
    Uracil & 0.056 & 0.056 & \textbf{0.051} & 0.062 & 0.066 \\
    \bottomrule
    Average \newline (8 molecules) & \textbf{0.083} & \textbf{0.083} & 0.086 & 0.086 & 0.088
    \end{tabular}
    \caption{Relative error $\tr(\mat{A} - \mat{F} \mat{F}^*) / \tr(\mat{A})$ of the low-rank approximation for each kernel matrix.
    Accelerated and simple \RPCholesky have the same error up to $\pm 0.001$ random variations.
    Best results are indicated in bold for each molecule.}
    \label{tab:rel_error}
\end{table}

\Cref{fig:pes} charts the average runtime needed to perform potential energy prediction with a low-rank preconditioner generated by accelerated \RPCholesky, simple \RPCholesky, ridge leverage score sampling (using the RRLS algorithm \cite{MM17}), uniform Nystr\"om approximation \cite{cutajar2016preconditioning,frangella2023randomized}, and greedy Nystr\"om approximation \cite{gardner2018gpytorch,wang2019exact}.
The results show that accelerated \RPCholesky achieves the fastest runtimes, improving on the other low-rank approximation algorithms by $4\%$ to $17\%$.
Accelerated \RPCholesky has the twin advantages that it quickly generates a preconditioner (fast prep time), and it leads to the minimal number of PCG iterations (fast PCG time).
Another important feature is the consistent performance across problems: accelerated \RPCholesky always produces near-optimal low-rank approximations, whereas other methods can struggle on difficult examples; see \cref{tab:rel_error}.

\begin{table}[t]
    \centering
    \begin{tabular}{llcc}
    \toprule
    Molecule & Formula & Old energy error & New energy error \\\midrule
    Aspirin & \ce{C_9H_8O_4} & 0.127 & \textbf{0.040} \\ 
    Benzene & \ce{C_6H_6} & 0.069 & \textbf{0.048} \\
    Ethanol & \ce{C_2H_6O} & \textbf{0.052} & 0.089 \\
    Malonaldehyde & \ce{C_3H_4O_2} & \textbf{0.074} & 0.133 \\ 
    Naphthalene & \ce{C_{10}H_8} & \textbf{0.113} & 0.126 \\ 
    Salicylic acid & \ce{C_7H_6O_3} & 0.105 & \textbf{0.102} \\
    Toluene & \ce{C_6H_5CH_3} & \textbf{0.092} & 0.146 \\ 
    Uracil & \ce{C_4H_4N_2O_2} & 0.103 & \textbf{0.010} \\
    \bottomrule
    \end{tabular}
    \caption{\label{tab:tab_results}
    Mean absolute error (kcal / mol) for the existing potential energy model \cite{chmiela2017machine} versus the new model \cref{eq:predictor}.
    Best results are indicated in bold for each molecule.}
\end{table}

The potential energy model in this paper is a direct application of standard KRR, yet it leads to higher accuracies for four of eight molecules than the more complicated kernel model in \cite{chmiela2017machine}.
On the most extreme example (uracil), the simple model is 10$\times$ more accurate then the existing model.
See \cref{tab:tab_results} for a table including all eight molecules.
These experiments demonstrate that simple kernel methods combined with scalable algorithms based on accelerated \RPCholesky can be competitive for problems in scientific machine learning.

\section{Theoretical results} \label{sec:theory}

This section compares accelerated \RPCholesky with simple \RPCholesky using rigorous error analysis.
On the surface, the comparison is straightforward.
Since accelerated \RPCholesky generates the same distribution of pivots as simple \RPCholesky, it satisfies the main \RPCholesky error bound \cite[Thm.~5.1]{CETW23}:

\begin{theorem}[Sufficient pivots for simple and accelerated \RPCholesky \cite{CETW23}] \label{thm:old_bound}
    Consider a psd matrix $\mat{A} \in \complex^{N \times N}$, and let $\mat{A}^{(k)}$ denote the random residual after applying simple or accelerated \RPCholesky with $k$ random pivots.
    Then
    \begin{equation*}
        \mathbb{E}[\tr(\mat{A}^{(k)})] \leq (1 + \varepsilon) \cdot \sum\nolimits_{i = r + 1}^N \lambda_i(\mat{A})
    \end{equation*}
    as soon as
    \begin{equation*}
        k \geq \frac{r}{\varepsilon} + r \log\Bigl(\frac{1}{\epsilon \eta}\Bigr),
        \quad \text{where } \eta = \frac{\sum_{i = r + 1}^N \lambda_i(\mat{A})}{\sum_{i = 1}^r \lambda_i(\mat{A})}.
    \end{equation*}
\end{theorem}

\Cref{thm:old_bound} bounds the sufficient number of pivots in accelerated \RPCholesky,
yet it does
not bound the sufficient number of iterations needed to produce the pivots via rejection sampling.
The following new result completes the theoretical understanding of accelerated \RPCholesky, by bounding the number of iterations. 
It also provides error bounds for block \RPCholesky.
This new result will be proved in \cref{sec:linear_algebra,sec:majorant,sec:permutation,sec:dynamical}.

\begin{theorem}[Sufficient iterations for accelerated and block \RPCholesky] \label{thm:main_bound}
    Consider a psd matrix $\mat{A} \in \complex^{N \times N}$, and let $\mat{A}^{(t)}$ denote the random residual after applying $t$ rounds of accelerated \RPCholesky or block \RPCholesky with a block size $b \geq 1$.
    Then
    \begin{equation*}
        \mathbb{E}[\tr(\mat{A}^{(t)})] \leq (1 + \varepsilon) \cdot \sum\nolimits_{i = r + 1}^N \lambda_i(\mat{A})
    \end{equation*}
    as soon as the number of proposed pivots $bt$ satisfies
    \begin{equation*}
        bt \geq \frac{r}{\varepsilon} + (r + b ) \log\Bigl(\frac{1}{\epsilon \eta}\Bigr),
        \quad \text{where } \eta = \frac{\sum_{i = r + 1}^N \lambda_i(\mat{A})}{\sum_{i = 1}^r \lambda_i(\mat{A})}.
    \end{equation*}
\end{theorem}

The comparison between accelerated and block \RPCholesky follows from a coupling argument in which block \RPCholesky accepts a greater number of pivots than accelerated \RPCholesky in each round, leading to a higher approximation quality.
However, in this comparison, accelerated \RPCholesky has the benefit of rejecting many of the proposed pivots and therefore producing a similarly accurate approximation with a smaller approximation rank.
When the methods are run with the same approximation rank, accelerated \RPCholesky can be much more accurate, as reflected in the numerical experiments of \cref{sec:numerics} and especially the failure modes of block \RPCholesky in \cref{fig:performance,fig:initial}.

The bounds in \cref{thm:main_bound,thm:old_bound} hold for all accuracy levels $\varepsilon > 0$,
and they each contain a term $r / \varepsilon$, which is the minimal number of pivots needed to guarantee an error level $(1 + \varepsilon) \cdot \sum\nolimits_{i = r + 1}^N \lambda_i(\mat{A})$ in any low-rank approximation produced by a pivoted partial Cholesky decomposition applied to a worst-case matrix \cite[Thm.~C.1]{CETW23}.
The bounds also contain a term $r \log(1 / (\varepsilon \eta))$ or $(r + b) \log(1 / (\varepsilon \eta))$, which represents the number of pivots that need to be accepted or proposed for the algorithm to improve the approximation quality from an initial error level $\tr(\mat{A}) = \eta^{-1} \cdot \sum\nolimits_{i = r + 1}^N \lambda_i(\mat{A})$
to a lower error level $(1 + \varepsilon) \cdot \sum\nolimits_{i = r + 1}^N \lambda_i(\mat{A})$.
Indeed, the only difference between \cref{thm:main_bound,thm:old_bound} is the quantity $b \log(1 / (\varepsilon \eta))$, present in the latter.
This difference becomes significant 
when the theorems 
are evaluated 
for a small target rank $r < b$ and a tiny approximation error $\eta \ll 1$.
This extreme setting occurs when the matrix $\mat{A}$ is nearly rank-$r$.
Then, 
accelerated \RPCholesky accepts very few proposals per round, so it takes many rounds to converge to high accuracy.

Last, \cref{thm:main_bound} can be compared with the bound of Deshpande, Rademacher, Vempala, and Wang \cite{DRVW06}.
Their bound was originally derived for block randomly pivoted QR, but it can be extended to block \RPCholesky.
Here is a simplified version of their result, adapted from \cite[sec.~3.6]{CETW23}:

\begin{theorem}[Sufficient iterations for block \RPCholesky with a large block size \cite{DRVW06}] \label{thm:drvw_bound}
    Fix $\varepsilon \le 3/2$ and a large block size $b \ge 3 r/\varepsilon$.
    Consider a psd matrix $\mat{A}$, and let $\mat{A}^{(t)}$ denote the random residual after applying $t$ rounds of block \RPCholesky.
    Then
    \begin{equation*}
        \mathbb{E}[\tr(\mat{A}^{(t)})] \leq (1 + \varepsilon) \cdot \sum\nolimits_{i = r + 1}^N \lambda_i(\mat{A})
    \end{equation*}
    as soon as
    \begin{equation*}
        t \geq 1 +  \frac{\log(1/\eta + 1)}{\log(3/\varepsilon)},
        \quad \text{where } \eta = \frac{\sum_{i = r + 1}^N \lambda_i(\mat{A})}{\sum_{i = 1}^r \lambda_i(\mat{A})}.
    \end{equation*}
\end{theorem}

When \cref{thm:main_bound,thm:drvw_bound} are both applicable, they specify a similar number of sufficient iterations for block \RPCholesky to converge to high accuracy.
However, 
\cref{thm:drvw_bound} has a more limited range of applicability since it requires a large block size $b \geq 3 r / \varepsilon$.
This prescription for a large block size is exactly the reverse of block \RPCholesky's behavior in practice, in which the method provides more accurate approximations given a smaller block size.
In particular,
\Cref{thm:drvw_bound}
has limited applicability in the typical setting where the user tunes the block size $b$ for a given computing setup to achieve the fastest runtime.
The final weakness of \cref{thm:drvw_bound} is that it does not extend to accelerated \RPCholesky.

\subsection{Method of proof} \label{sec:method}

The proof uses many probabilistic and linear algebraic tools that were introduced in the analysis of simple \RPCholesky \cite[Thm.~5.1]{CETW23}, and it follows the same core approach: 
\begin{itemize}
    \item[(a)] Evaluate the function that gives the expected value of the residual matrix after generating and processing a single pivot.
    \item[(b)] Prove monotonicity and concavity properties for the expected residual function. 
    \item[(c)] Use monotonicity and concavity to identify a worst-case matrix $\mat{A} \in \mathbb{R}^{N \times N}$.
    \item[(d)] Analyze the worst-case matrix to produce error bounds.
\end{itemize} 

As the core difficulty, accelerated \RPCholesky leads to an expected residual function that does not satisfy the matrix monotonicity property.
Therefore, it is necessary to develop more subtle concavity and monotonicity properties and the proof is twice as long as before.

The expected residual function for accelerated \RPCholesky can be derived as follows.
Let $\Ahat^{(1, i)}$ be the low-rank approximation generated from the first $i\le b$ proposals, and let $\mat{A}^{(1, i)} = \mat{A} - \vphantom{\mat{\hat{A}}}\Ahat^{(1, i)}$ be the corresponding residual matrix.
Accelerated \RPCholesky updates the residual matrix in three steps:
\begin{enumerate}
\item A pivot $s_{i+1}'$ is sampled according to $\mathbb{P}\{s_{i+1}' = j\} = \mat{A}(j,j) / \tr \mat{A}$.
\item The pivot is accepted with probability $\mat{A}^{(1, i)}(j,j) / \mat{A}(j,j)$.
\item Conditional on acceptance, the residual matrix is updated:
\begin{equation*}
    \mat{A}^{(1, i+1)} = \mat{A}^{(1, i)} - \mat{A}^{(1, i)}(:, j) \mat{A}^{(1, i)}(j, :) / \mat{A}^{(1, i)}(j,j).
\end{equation*}
\end{enumerate}
By combining steps 1--3 above, the expected residual matrix is:
\begin{align*}
    \mathbb{E}[\mat{A}^{(1,i+1)}| \mat{A}^{(1,i)}]
    &= \mat{A}^{(1,i)} - \sum\nolimits_{j=1}^N \frac{\mat{A}(j,j)}{\tr(\mat{A})} \cdot \frac{\mat{A}^{(1, i)}(j,j)}{\mat{A}(j,j)} \cdot \frac{\mat{A}^{(1,i)}(:,j) \mat{A}^{(1,i)}(j, :)}{\mat{A}^{(1, i)}(j,j)} \\
    &= \mat{A}^{(1,i)} - \frac{1}{\tr(\mat{A})} (\mat{A}^{(1,i)})^2.
\end{align*}
Therefore, the expected residual function can be defined as:
\begin{equation*}
    \phi_{\alpha}(x) = x - \tfrac{1}{\alpha} x^2,
    \quad \text{where } \alpha = \tr(\mat{A}) \text{ and } x = \mat{A}^{(1,i)}.
\end{equation*}
This function depends on a scalar parameter $\alpha > 0$, and it can be evaluated for inputs $x$ that are scalars or matrices.
By repeated function composition, this function bounds the expected residual matrix of accelerated \RPCholesky after any number of steps.

Bounding the error in this way requires an expected residual function that is concave and monotone.
However, the function $\phi_{\alpha}(x)$ barely has enough monotonocity for the argument to work.
In particular, the mapping $x \mapsto \phi_{\alpha}(x)$ is only monotone for scalars $x \in [0, \alpha/2]$, and it fails to be monotone for large classes of psd matrices.

The analysis pushes through the technical difficulties and produces an upper bound for accelerated \RPCholesky depending on the expected residual function.
The bound becomes worse when the eigenvalues are averaged together or increased.
Hence, the worst-case matrix takes the form
\begin{equation*}
    \mat{A}
    = \operatorname{diag}\Bigl\{\underbrace{\frac{\gamma}{r}, \ldots, \frac{\gamma}{r}}_{r \text{ times}},
    \underbrace{\frac{\beta}{N - r}, \ldots, \frac{\beta}{N - r}}_{N - r \text{ times}},\Bigr\},
    \quad \text{where } \begin{cases}
        \gamma = \sum_{i = 1}^r \lambda_i(\mat{A}), \\
        \beta = \sum_{i = r + 1}^N \lambda_i(\mat{A}).
    \end{cases}
\end{equation*}
It is a two-cluster matrix with a small block of leading eigenvalues and a large block of trailing eigenvalues.
The identification of this two-cluster matrix is not new: the same worst-case matrix was already encountered in the proof of \cite[Thm.~5.1]{CETW23}.
Applying the expected residual function to this matrix completes the proof of \cref{thm:main_bound}.

\subsection{Organization of proof}

Here is the overall organization for the proof of \cref{thm:main_bound}, which will be carried out over the next four subsections.

\paragraph{Part I: Properties of the expected residual function}
The first part of the analysis (\cref{sec:linear_algebra}) establishes concavity and monotonicity properties for the expected residual function.
The concavity and monotonicity properties are more detailed than what was proved in \cite[Thm.~5.1]{CETW23}.

\paragraph{Part II: Accelerated \RPCholesky error}
The second part of the analysis (\cref{sec:majorant}) establishes the following error bound:
\begin{lemma}[Central error bound]
    \label{prop:accelerated}
    Fix a block size $b \geq 1$ and a target psd matrix $\mat{A} \in \complex^{N \times N}$.
    Let $\mat{A}^{(t)}$ denote the residual matrix after applying $t$ rounds of accelerated or block \RPCholesky with blocks of $b$ proposals.
    Also introduce the matrix-valued function
    \begin{equation*}
        \mat{\Phi}_b(\mat{M}) = \underbrace{\phi_{\tr(\mat{M})} \circ \phi_{\tr(\mat{M})} \circ \cdots \circ \phi_{\tr(\mat{M})}}_{b  \text{ times}}(\mat{M}), \quad \text{where} \quad \phi_\alpha(x) = x - \tfrac{1}{\alpha} x^2,
    \end{equation*}
    and let $\mat{\Phi}^{(t)}_b = \mat{\Phi}_b \circ \mat{\Phi}_b \circ \cdots \circ \mat{\Phi}_b$
    denote the $t$-fold composition of $\mat{\Phi}_b$.
    Then, the eigenvalues of $\expect[\mat{A}^{(t)}]$ are bounded by the eigenvalues of $\mat{\Phi}_b^{(t)}(\mat{A})$ as follows:
    \begin{equation*}
        \lambda_i(\expect[\mat{A}^{(t)}]) \leq \lambda_i(\mat{\Phi}_b^{(t)}(\mat{A}))
        \quad \text{for each } i = 1, \ldots, N.
    \end{equation*}
\end{lemma}
The proof of \cref{prop:accelerated} is the most intricate argument of the paper.
This lemma shows that the expected trace norm error is bounded by a deterministic quantity: 
\begin{equation*}
    \expect \tr(\mat{A}^{(t)}) \leq \tr(\mat{\Phi}_b^{(t)}(\mat{A})).
\end{equation*}
This upper bound depends only on the eigenvalues of $\mat{A}$, independent of the eigenvectors.

\paragraph{Part III: Permutation averaging lemma}
The next key idea is to average together the leading eigenvalues of $\mat{A}$ and also average together the trailing eigenvalues.
\Cref{sec:permutation} proves the following helpful lemma, which shows that the averaging can only make the error worse:
\begin{lemma}[Permutation averaging] \label{lem:permutation}
    For any diagonal, psd matrix $\mat{\Lambda} \in \mathbb{R}^{N \times N}$ and any random permutation $\mat{P} \in \mathbb{R}^{N \times N}$,
    \begin{equation*}
        \tr(\mat{\Phi}_b^{(t)}(\mat{\Lambda}))
        \leq \tr(\mat{\Phi}_b^{(t)}(\expect[\mat{P} \mat{\Lambda} \mat{P}^*])).
    \end{equation*}
\end{lemma}

\paragraph{Part IV: Main error bound}
The last part of the analysis in \cref{sec:dynamical} examines the worst-case matrix $\mat{A}$ and obtains a simple recursive formula for $\tr(\mat{\Phi}_b^{(t)}(\mat{A}))$ in terms of $t$.
Bounding the recursion leads to the proof of the main result, \cref{thm:main_bound}.

\subsection{Part I: Properties of the expected residual function} \label{sec:linear_algebra}

This section establishes concavity and monotonicity properties for the expected residual function.
The first lemma considers the application of this function to scalars:

\begin{lemma}[Expected residual properties] \label{lem:properties}
    Let $\phi_{\alpha}: [0, \alpha] \rightarrow [0, \alpha]$ be the function $\phi_{\alpha}(x) = x - \frac{1}{\alpha} x^2$, and let
    \begin{equation*}
        \phi_{\alpha}^{(b)} = \underbrace{\phi_{\alpha} \circ \phi_{\alpha} \circ \cdots \circ \phi_{\alpha}}_{b  \text{ times}}
    \end{equation*}
    denote the $b$-fold composition of $\phi_{\alpha}$. 
    Then the following holds for each $b \geq 1$:
    \begin{itemize}
        \item[(a)] \textbf{Uniform boundedness:} $\phi_{\alpha}^{(b)} \le \alpha / 4$.
        \item[(b)] \textbf{Monotonicity in $x$ on $[0,\alpha/2]$:} $\phi_{\alpha}^{(b)}(x) \leq \phi_{\alpha}^{(b)}(x')$ if $x \le x' \le \alpha/2$.
        \item[(c)] \textbf{Joint positive homogeneity:} $\phi_{\beta \cdot \alpha}^{(b)}(\beta \cdot x) 
            = \beta \cdot \phi_{\alpha}^{(b)}(x)$ for $\beta > 0$.
        \item[(d)] \textbf{Joint concavity:} The mapping $(x, \alpha) \mapsto \phi_{\alpha}^{(b)}(x)$ is concave:
        \begin{equation*}
            \theta \phi_{\alpha_1}^{(b)}(x_1) + (1 - \theta) \phi_{\alpha_2}^{(b)}(x_2) \leq \phi_{\theta \alpha_1 + (1-\theta) \alpha_2}^{(b)}(\theta x_1 + (1 - \theta) x_2)
            \quad \text{for} \quad
            \theta \in (0, 1).
        \end{equation*}
        \item[(e)] \textbf{Upper bound:} $\phi_{\alpha}^{(b)}(x) \le 1 / (b\alpha^{-1} + x^{-1})$.
    \end{itemize}
\end{lemma}

\begin{proof} See \cref{sec:deferred-proofs}. \end{proof}

The next lemma shows that the expected residual function continues to exhibit concavity and monotonicity when it is applied to certain classes of matrices.

\begin{lemma}[Expected residual properties for matrices] \label{lem:matrix}
Define
\begin{equation*}
    \mat{\Phi}_b(\mat{A}) = \underbrace{\phi_{\tr(\mat{A})} \circ \phi_{\tr(\mat{A})} \circ \cdots \circ \phi_{\tr(\mat{A})}}_{b  \text{ times}}(\mat{A}), \quad \text{where} \quad \phi_\alpha(x) = x - \tfrac{1}{\alpha} x^2.
\end{equation*}
Let $\mat{\Phi}^{(t)}_b = \mat{\Phi}_b \circ \mat{\Phi}_b \circ \cdots \circ \mat{\Phi}_b$
denote the $t$-fold composition of $\mat{\Phi}_b$.
Then the following holds for any $t \geq 1$:
\begin{itemize}
    \item[(a)] $\mat{\Phi}^{(t)}_b$ is non-decreasing on the space of diagonal, psd matrices:
    \begin{equation*}
        \mat{\Phi}^{(t)}_b(\mat{\Lambda}_1) \preceq \mat{\Phi}^{(t)}_b(\mat{\Lambda}_2)
        \quad \text{if } \mat{\Lambda}_1 \preceq \mat{\Lambda}_2.
    \end{equation*}
    \item[(b)] $\mat{\Phi}^{(t)}_b$ is concave on the space of diagonal, psd matrices:
    \begin{equation*}
        \theta \mat{\Phi}^{(t)}_b(\mat{\Lambda}_1) 
        + (1 - \theta) \mat{\Phi}^{(t)}_b(\mat{\Lambda}_2)
        \preceq
        \mat{\Phi}^{(t)}_b( \theta \mat{\Lambda}_1 + (1-\theta) \mat{\Lambda}_2
        )
        \quad \text{for } \theta \in (0, 1).
    \end{equation*}
    \item[(c)] $\mat{\Phi}_1$ is concave on the space of all psd matrices:
    \begin{equation*}
        \theta \mat{\Phi}_1(\mat{A}_1) 
        + (1 - \theta) \mat{\Phi}_1(\mat{A}_2)
        \preceq
        \mat{\Phi}_1( \theta \mat{A}_1 + (1-\theta) \mat{A}_2
        )
        \quad \text{for } \theta \in (0, 1).
    \end{equation*}
\end{itemize}
\end{lemma}

\begin{proof} See \cref{sec:deferred-proofs}. \end{proof}

\subsection{Part II: Accelerated \RPCholesky error} \label{sec:majorant}

This section analyzes the error of accelerated \RPCholesky and proves \cref{prop:accelerated}.
The technical crux is the following proposition, which bounds the error of accelerated \RPCholesky or block \RPCholesky applied to a random psd matrix.

\begin{proposition}[Block and accelerated \RPCholesky on a random matrix]
    \label{prop:random}
    Consider a random psd matrix $\mat{A} \in \complex^{N \times N}$, and let $\mat{A}^{(1, b)}$ denote the residual matrix after applying one round of accelerated \RPCholesky or block \RPCholesky to $\mat{A}$ with $b \geq 1$ proposals.
    Also introduce the matrix-valued function
    \begin{equation*}
        \mat{\Phi}_b(\mat{M}) = \underbrace{\phi_{\tr(\mat{M})} \circ \phi_{\tr(\mat{M})} \circ \cdots \circ \phi_{\tr(\mat{M})}}_{b  \text{ times}}(\mat{M}), \quad \text{where} \quad \phi_\alpha(x) = x - \tfrac{1}{\alpha} x^2,
    \end{equation*}
    Then the eigenvalues of $\expect[\mat{A}^{(1, b)}]$ are bounded by the eigenvalues of $\mat{\Phi}_b^{(1, b)}(\expect[\mat{A}])$:
    \begin{equation*}
        \lambda_i(\expect[\mat{A}^{(1, b)}]) \leq \lambda_i(\mat{\Phi}_b(\expect[\mat{A}]))
        \quad \text{for each } i = 1, \ldots, N.
    \end{equation*}
\end{proposition}

The proof of this result relies on the min--max principle \cite[Thm.~8.9]{Zha11}:

\begin{lemma}[Min--max principle]
   The eigenvalues of a Hermitian matrix $\mat{M} \in \complex^{N \times N}$ satisfy
    \begin{equation*}
        \lambda_i(\mat{M}) =
        \min_{\operatorname{dim}(\mathcal{S}) = N - i + 1}\,
        \max_{\vec{v} \in \mathcal{S},\,
        \lVert \vec{v} \rVert = 1} \vec{v}^* \mat{M} \vec{v},
    \end{equation*}
    where the maximum is taken over all $(N-i+1)$-dimensional linear subspaces $\mathcal{S}$.
    The optimal $\mathcal{S}$ is spanned by the eigenvectors of $\mat{M}$ with the $N - i + 1$ smallest eigenvalues.
\end{lemma}

\begin{proof}[Proof of \cref{prop:random}]
    The proof has three steps.
    
    \textbf{\textit{Step 1: Comparison of block and accelerated \RPCholesky.}} Accelerated \RPCholesky and block \RPCholesky yield the same distribution of pivot proposals, so the proposals $\set{S}' = \{s_1', \ldots, s_b'\}$ are assumed to be identical.
    Block \RPCholesky then accepts all the proposals, leading to a residual matrix
    \begin{equation*}
        \mat{\tilde{A}}^{(1, b)} = \mat{A} - \mat{A}(:, \set{S}') \mat{A}(\set{S}', \set{S}')^{\dagger} \mat{A}(\set{S}', :)
        = \mat{A}^{1/2} (\mathbf{I} - \mat{\Pi}_{\mat{A}^{1/2}(:, \set{S}')}) \mat{A}^{1/2},
    \end{equation*}
    where
    \begin{equation*}
        \mat{\Pi}_{\mat{A}^{1/2}(:, \set{S}')}
        = \mat{A}^{1/2}(:, \set{S}') \mat{A}(\set{S}', \set{S}')^{\dagger} \mat{A}^{1/2}(\set{S}', :)
    \end{equation*}
    is the orthogonal projection onto the range of $\mat{A}^{1/2}(:, \set{S}')$.
    In contrast, accelerated \RPCholesky accepts a subset of the proposals $\set{S} \subseteq \set{S}'$, leading to a residual matrix
    \begin{equation*}
        \mat{A}^{(1, b)} 
        = \mat{A}^{1/2} (\mathbf{I} - \mat{\Pi}_{\mat{A}^{1/2}(:, \set{S})}) \mat{A}^{1/2}.
    \end{equation*}
    The range of $\mat{A}^{1/2}(:, \set{S}')$ is no smaller than the range of $\mat{A}^{1/2}(:, \set{S})$, so 
    \begin{equation}
    \label{eq:to_conjugate}
        \mathbf{I} - \mat{\Pi}_{\mat{A}^{1/2}(:, \set{S}')} 
        \preceq \mathbf{I} - \mat{\Pi}_{\mat{A}^{1/2}(:, \set{S})}.
    \end{equation}
    Conjugate both sides of \cref{eq:to_conjugate} with $\mat{A}^{1/2}$ to yield
    \begin{equation*}
        \mat{\tilde{A}}^{(1, b)} 
        = \mat{A}^{1/2} (\mathbf{I} - \mat{\Pi}_{\mat{A}^{1/2}(:, \set{S}')}) \mat{A}^{1/2} 
        \preceq \mat{A}^{1/2} (\mathbf{I} - \mat{\Pi}_{\mat{A}^{1/2}(:, \set{S})}) \mat{A}^{1/2}
        = \mat{A}^{(1, b)}.
    \end{equation*}
    Taking expectations, it follows that $\mathbb{E}[\mat{\tilde{A}}^{(1, b)}] \preceq \mathbb{E}[\mat{A}^{(1, b)}]$
    and, by Weyl's inequality,
    \begin{equation*}
        \lambda_i(\mathbb{E}[\mat{\tilde{A}}^{(1, b)}]) \leq \lambda_i(\mathbb{E}[\mat{A}^{(1, b)}]), \quad \text{for each } i=1, \ldots, N.
    \end{equation*}
    This completes the comparison of block and accelerated \RPCholesky.

    \textbf{\emph{Step 2: Bounding $\vec{v}^* \mathbb{E}[\mat{A}^{(1, b)}] \vec{v}$.}}
    Now it suffices to analyze the accelerated \RPCholesky error.
    This proof will do so by (a) deriving an upper bound for the quadratic form $\vec{v}^* \mathbb{E}[\mat{A}^{(1, b)}] \vec{v}$, (b) constructing a $(N - i + 1)$-dimensional subspace where the Rayleigh quotients are no larger than $\lambda_i(\mat{\Phi}_b(\expect[\mat{A}]))$, and (c) invoking the min--max principle.
    
    The residual for accelerated \RPCholesky after the first proposal has expectation
    \begin{equation*}
        \mathbb{E}[\mat{A}^{(1, 1)}]
        = \mathbb{E}[\mathbb{E}[\mat{A}^{(1, 1)}|\mat{A}]]
        = \mathbb{E}[\mat{A} - \tfrac{1}{\tr(\mat{A})} \mat{A}^2]
        = \mathbb{E}[\mat{\Phi}_1(\mat{A})].
    \end{equation*}
    Since $\mat{\Phi}_1$ is concave over psd matrices (\cref{lem:matrix}c), Jensen's inequality implies
    \begin{equation*}
    \label{eq:conjugate_me}
        \mathbb{E}[\mat{A}^{(1, 1)}]
        = \mathbb{E}[\mat{\Phi}_1(\mat{A})]
        \preceq \mat{\Phi}_1(\mathbb{E}[\mat{A}]).
    \end{equation*}
    Next fix a unit-length vector $\vec{v} \in \mathbb{C}^N$ and conjugate both sides of this display with $\vec{v}$:
    \begin{equation*}
        \mathbb{E}[\vec{v}^* \mat{A}^{(1, 1)} \vec{v}]
        \leq \vec{v}^* \mat{\Phi}_1(\mathbb{E}[\mat{A}]) \vec{v}.
    \end{equation*}
    By boundedness of $\phi_{\mathbb{E}[\tr(\mat{A})]}$ and 
    monotonicity of $\phi_{\mathbb{E}[\tr(\mat{A})]}^{(b-1)}$ (\cref{lem:properties}a--b), it follows
    \begin{equation}
    \label{eq:exploit_me}
        \phi^{(b-1)}_{\mathbb{E}[\tr(\mat{A})]}(\mathbb{E}[\vec{v}^* \mat{A}^{(1, 1)} \vec{v}])
        \leq \phi^{(b-1)}_{\mathbb{E}[\tr(\mat{A})]}(\vec{v}^* \mat{\Phi}_1(\mathbb{E}[\mat{A}]) \vec{v}).
    \end{equation}
    
    The residual matrix after the $(\ell + 1)$st proposal has conditional expectation
    \begin{equation*}
    \label{eq:same_calc}
        \mathbb{E}[\mat{A}^{(1, \ell+1)}| \mat{A}^{(1, \ell)}, \mat{A}]
        = \mat{A}^{(1, \ell)} - \tfrac{1}{\tr(\mat{A})} (\mat{A}^{(1, \ell)})^2.
    \end{equation*}
    Conjugate both sides of this display with $\vec{v}$ to obtain:
    \begin{equation*}
    \begin{aligned}
        \mathbb{E}[\vec{v}^* \mat{A}^{(1, \ell+1)} \vec{v} | \mat{A}^{(1, \ell)}, \mat{A}]
        &= \vec{v}^* \mat{A}^{(1, \ell)} \vec{v} - \tfrac{1}{\tr(\mat{A})} \vec{v}^* (\mat{A}^{(1, \ell)})^2 \vec{v} \\
        &\leq \vec{v}^* \mat{A}^{(1, \ell)} \vec{v} - \tfrac{1}{\tr(\mat{A})} (\vec{v}^* \mat{A}^{(1, \ell)} \vec{v})^2
        = \phi_{\tr(\mat{A})}(\vec{v}^* \mat{A}^{(1, \ell)} \vec{v}).
    \end{aligned}
    \end{equation*}
    The inequality follows because $\vec{v} \vec{v}^*$ is a contraction.
    Take expectations and apply Jensen's inequality to the concave function $x \mapsto \phi_{\tr(\mat{A})}(x)$ (\cref{lem:properties}d):
    \begin{equation*}
        \mathbb{E}[\vec{v}^* \mat{A}^{(1, \ell+1)} \vec{v} | \mat{A}]
        \leq \mathbb{E}[\phi_{\tr(\mat{A})}(\vec{v}^* \mat{A}^{(1, \ell)} \vec{v}) | \mat{A}]
        \leq \phi_{\tr(\mat{A})}(\mathbb{E}[\vec{v}^* \mat{A}^{(1, \ell)} \vec{v} | \mat{A}]).
    \end{equation*}
    Use the boundedness of $\phi_{\tr(\mat{A})}$ and monotonicity of $\phi_{\tr(\mat{A})}^{(b-\ell-1)}$
    (\cref{lem:properties}a--b) to show
    \begin{equation*}
        \phi^{(b-\ell-1)}_{\tr(\mat{A})}(\mathbb{E}[\vec{v}^* \mat{A}^{(1, \ell+1)} \vec{v} | \mat{A}])
        \leq \phi^{(b-\ell)}_{\tr(\mat{A})}(\mathbb{E}[\vec{v}^* \mat{A}^{(1, \ell)} \vec{v} | \mat{A}]).
    \end{equation*}
    Chaining the above inequalities for $\ell = 1, 2, \ldots, b - 1$ yields
    \begin{equation*}
        \mathbb{E}[\vec{v}^* \mat{A}^{(1, b)} \vec{v} | \mat{A}]
        \leq \phi^{(b-1)}_{\tr(\mat{A})}(\mathbb{E}[\vec{v}^* \mat{A}^{(1, 1)} \vec{v} | \mat{A}]).
    \end{equation*}
    Taking expectations yields:
    \begin{equation}
    \label{eq:use_me}
    \begin{aligned}
        \vec{v}^* \mathbb{E}[\mat{A}^{(1, b)}] \vec{v}
        &\leq \mathbb{E}[\phi^{(b-1)}_{\tr(\mat{A})}(\mathbb{E}[\vec{v}^* \mat{A}^{(1, 1)} \vec{v} | \mat{A}])] \\
        &\leq \phi^{(b-1)}_{\mathbb{E}[\tr(\mat{A})]}(\mathbb{E}[\vec{v}^* \mat{A}^{(1, 1)} \vec{v}]) \leq \phi^{(b-1)}_{\mathbb{E}[\tr(\mat{A})]}(\vec{v}^* \mat{\Phi}_1(\mathbb{E}[\mat{A}]) \vec{v}).
    \end{aligned}
    \end{equation}
    The second inequality is Jensen's inequality applied to the jointly concave function $(x, \alpha) \mapsto \phi^{(b-1)}_\alpha(x)$ (\cref{lem:properties}d), and the third inequality is \cref{eq:exploit_me}.
    This establishes a convenient upper bound for the quadratic form $\vec{v}^* \mathbb{E}[\mat{A}^{(1, b)}] \vec{v}$.

    \textbf{\emph{Step 3: Apply the min--max principle.}}
    Let $\vec{v}_1, \ldots, \vec{v}_N$ be eigenvectors of $\mat{\Phi}_1(\expect[\mat{A}])$, associated with its decreasingly ordered eigenvalues.
    For any unit-length vector $\vec{v} \in \operatorname{span}\{\vec{v}_i, \ldots, \vec{v}_N\}$, the min--max variational principle implies
    \begin{equation*}
        \vec{v}^* \mat{\Phi}_1(\mathbb{E}[\mat{A}]) \vec{v} 
        \leq \lambda_i(\mat{\Phi}_1(\mathbb{E}[\mat{A}])).
    \end{equation*}   
    By \cref{eq:use_me}, boundedness of $\phi_{\mathbb{E}[\tr(\mat{A})]}$, and 
    monotonicity of $\phi_{\mathbb{E}[\tr(\mat{A})]}^{(b-1)}$ (\cref{lem:properties}a--b), it follows
    \begin{equation*}
        \vec{v}^* \mathbb{E}[\mat{A}^{(1, b)}] \vec{v}
        \leq \phi^{(b-1)}_{\mathbb{E}[\tr(\mat{A})]}(\vec{v}^* \mat{\Phi}_1(\mathbb{E}[\mat{A}]) \vec{v})
        \leq \phi^{(b-1)}_{\mathbb{E}[\tr(\mat{A})]}(\lambda_i(\mat{\Phi}_1(\mathbb{E}[\mat{A}]))).
    \end{equation*}
    By boundedness of $\phi_{\mathbb{E}[\tr(\mat{A})]}$ and the 
    monotonicity of $\phi_{\mathbb{E}[\tr(\mat{A})]}^{(b-1)}$ (\cref{lem:properties}a--b),
    the matrices $\phi^{(b-1)}_{\expect[\tr(\mat{A})]}(\mat{\Phi}_1(\expect[\mat{A}]))$ and $\mat{\Phi}_1(\expect[\mat{A}])$ have the same eigenvectors with the same ordering of the eigenvalues.
    Hence,
    \begin{equation*}
        \phi^{(b-1)}_{\mathbb{E}[\tr(\mat{A})]}(\lambda_i(\mat{\Phi}_1(\mathbb{E}[\mat{A}]))) 
        = \lambda_i( \phi^{(b-1)}_{\mathbb{E}[\tr(\mat{A})]}(\mat{\Phi}_1(\mathbb{E}[\mat{A}])))
        = \lambda_i(\mat{\Phi}_b(\mathbb{E}[\mat{A}])).
    \end{equation*}
    Combining the two previous displays yields
    \begin{equation*}
        \vec{v}^* \expect[\mat{A}^{(1, b)}] \vec{v} \leq \lambda_i \bigl( \mat{\Phi}_b(\mathbb{E}[\mat{A}]) \bigr)
        \quad \text{for any unit-length } \vec{v} \in \operatorname{span}\{\vec{v}_i, \ldots, \vec{v}_N\}.
    \end{equation*}
    It follows from the min-max variational principle that $\lambda_i(\expect[\mat{A}^{(1, b)}]) \leq \lambda_i(\mat{\Phi}_b(\expect[\mat{A}]))$, completing the proof.
\end{proof}

The central result of this section is proved by iteratively applying \cref{prop:random}.

\begin{proof}[Proof of \Cref{prop:accelerated}]
    The $(s+1)$st round of accelerated \RPCholesky or block \RPCholesky leads to the same residual matrix as a single round of accelerated \RPCholesky or block \RPCholesky applied to the psd matrix $\mat{A}^{(s)}$.
    Thus, \Cref{prop:random} shows that
    \begin{equation}
    \label{eq:eigenvalue_condition}
        \lambda_i(\expect[\mat{A}^{(s+1)}]) \leq \lambda_i(\mat{\Phi}_b(\expect[\mat{A}^{(s)}]))
        \quad \text{for each } i = 1, \ldots, N.
    \end{equation}
    Next take eigendecompositions
    \begin{equation*}
        \expect[\mat{A}^{(s+1)}] 
        = \mat{Q}_1 \mat{\Lambda}_1 \mat{Q}_1^*
        \quad \text{and} \quad
        \mat{\Phi}_b(\expect[\mat{A}^{(s)}])
        = \mat{Q}_2 \mat{\Lambda}_2 \mat{Q}_2^*,
    \end{equation*}
    where the eigenvalues are listed in weakly decreasing order.
    \Cref{eq:eigenvalue_condition} implies $\mat{\Lambda}_1 \preceq \mat{\Lambda}_2$, so the monotonicity of $\mat{\Phi}_b^{(t-s-1)}$ (\cref{lem:matrix}a) further implies
    \begin{equation}
    \label{eq:psd_ordering}
        \mat{\Phi}^{(t-s-1)}_b(\mat{\Lambda}_1)
        \preceq \mat{\Phi}^{(t-s-1)}_b(\mat{\Lambda}_2).
    \end{equation}
    The function $\mat{\Phi}_b$ is covariant under unitary basis transformations:
    \begin{equation*}
        \mat{\Phi}_b(\mat{Q} \mat{\Lambda} \mat{Q}^*)
        = \mat{Q} \mat{\Phi}_b(\mat{\Lambda}) \mat{Q}^*
        \quad \text{for unitary } \mat{Q} \in \mathbb{C}^{N \times N}.
    \end{equation*}
    Therefore, the eigenvalue ordering \cref{eq:psd_ordering} implies
    \begin{equation*}
    \begin{aligned}
        \mat{Q}_1^* \mat{\Phi}^{(t-s-1)}_b(\expect[\mat{A}^{(s+1)}]) \mat{Q}_1
        &= \mat{\Phi}^{(t-s-1)}_i(\mat{\Lambda}_1) \\
        &\preceq
        \mat{\Phi}^{(t-s-1)}_i(\mat{\Lambda}_2) = \mat{Q}_2^* \mat{\Phi}^{(t-s)}_b(\expect[\mat{A}^{(s)}]) \mat{Q}_2.
    \end{aligned}
    \end{equation*}
    Since the eigenvalues are invariant under unitary basis transformations, Weyl's inequality implies
    \begin{equation}
    \label{eq:iteration}
    \begin{aligned}
        \lambda_i(\mat{\Phi}^{(t-s-1)}_b(\expect[\mat{A}^{(s+1)}]))
        &=
        \lambda_i(\mat{Q}_1^* \mat{\Phi}^{(t-s-1)}_b(\expect[\mat{A}^{(s+1)}]) \mat{Q}_1) \\
        &\leq \lambda_i(\mat{Q}_2^* \mat{\Phi}^{(t-s)}_i(\expect[\mat{A}^{(s)}]) \mat{Q}_2) = \lambda_i(\mat{\Phi}^{(t-s)}_b(\expect[\mat{A}^{(s)}])).
    \end{aligned}
    \end{equation}
    Apply \cref{eq:iteration} for $s = 0, 1, \ldots, t-1$ to complete the proof.
\end{proof}

\subsection{Part III: Permutation averaging lemma}
\label{sec:permutation}

The permutation averaging lemma has the following short and simple proof.

\begin{proof}[Proof of \cref{lem:permutation}]
    Since the trace is invariant under permutations
    and the function $\mat{\Phi}_b$ is covariant under permutations,
    \begin{equation*}
        \tr(\mat{\Phi}_b^{(t)}(\mat{\Lambda}))
        = \tr(\mat{P}\mat{\Phi}_b^{(t)}(\mat{\Lambda})\mat{P}^*)
        = \tr(\mat{\Phi}_b^{(t)}(\mat{P} \mat{\Lambda} \mat{P}^*)).
    \end{equation*}
    Now apply Jensen's inequality to the concave function $\mat{\Phi}_b^{(t)}$ (\cref{lem:matrix}b):
    \begin{equation*}
        \expect[\mat{\Phi}_b^{(t)}(\mat{P} \mat{\Lambda} \mat{P}^*)]
        \preceq \mat{\Phi}_b^{(t)}(\expect[\mat{P} \mat{\Lambda} \mat{P}^*]).
    \end{equation*}
    Since the trace respects the psd ordering, it follows
    \begin{equation*}
        \tr(\mat{\Phi}_b^{(t)}(\mat{\Lambda}))
        = \tr(\expect[\mat{\Phi}_b^{(t)}(\mat{P} \mat{\Lambda} \mat{P}^*)])
        \leq \tr(\mat{\Phi}_b^{(t)}(\expect[\mat{P} \mat{\Lambda} \mat{P}^*])).
    \end{equation*}
    This completes the proof.
\end{proof}

\subsection{Part IV: Main error bound} \label{sec:dynamical}

At this point, the main error bound for accelerated \RPCholesky is nearly in hand.
The last step is an analysis argument that explicitly bounds the error for a worst-case instance of accelerated \RPCholesky.

\begin{proof}[Proof of \cref{thm:main_bound}]
    Take an eigendecomposition $\mat{A} = \mat{Q} \mat{\Lambda} \mat{Q}^*$. By \cref{prop:accelerated},
    \begin{equation*}
        \mathbb{E}[\tr(\mat{A}^{(t)})] \leq \tr(\mat{\Phi}_b^{(t)}(\mat{A}))
        = \tr(\mat{\Phi}_b^{(t)}(\mat{Q} \mat{\Lambda} \mat{Q}^*)).
    \end{equation*}
    Since the trace is invariant under unitary basis transformations
    and the function $\mat{\Phi}_b$ is covariant under unitary basis transformations,
    \begin{equation*}
        \tr(\mat{\Phi}_b^{(t)}(\mat{Q} \mat{\Lambda} \mat{Q}^*))
        = \tr(\mat{Q} \mat{\Phi}_b^{(t)}( \mat{\Lambda} )\mat{Q}^*)
        = \tr(\mat{\Phi}_b^{(t)}( \mat{\Lambda} )).
    \end{equation*}
    Therefore, the upper bound only depends on the eigenvalue matrix $\mat{\Lambda} \in \mathbb{R}^{N \times N}$.

    Next, take a random permutation $\mat{P}$ that permutes the first $r$ coordinates uniformly at random and permutes the trailing $N - r$ coordinates uniformly at random.
    By \cref{lem:permutation},
    \begin{equation*}
        \tr(\mat{\Phi}_b^{(t)}(\mat{\Lambda}))
        \leq \tr(\mat{\Phi}_b^{(t)}(\mat{\hat{\Lambda}})),
        \quad \text{where } \mat{\hat{\Lambda}} = \expect[\mat{P} \mat{\Lambda} \mat{P}^*].
    \end{equation*}
    The matrix $\mat{\hat{\Lambda}}$ has the explicit form
    \begin{equation*}
        \mat{\hat{\Lambda}}
        = \operatorname{diag}\Bigl\{\underbrace{\frac{\alpha}{r}, \ldots, \frac{\alpha}{r}}_{r \text{ times}},
        \underbrace{\frac{\beta}{N - r}, \ldots, \frac{\beta}{N - r}}_{N - r \text{ times}},\Bigr\},
        \quad \text{where } \begin{cases}
            \alpha = \sum_{i = 1}^r \lambda_i(\mat{A}), \\
            \beta = \sum_{i = r + 1}^N \lambda_i(\mat{A}).
        \end{cases}
    \end{equation*}
    It is a two-cluster matrix that leads to the worst possible value of $\tr(\mat{\Phi}_b^{(t)}(\mat{\Lambda}))$

    \newcommand{\Lambdahat}{\smash{\mat{\hat{\Lambda}}}}
    For the next step, argue by induction that
    \begin{equation}
    \label{eq:comparison}
        \mat{\Phi}_b^{(t)}(\Lambdahat)
        \preceq \Lambdahat^{(t)},
        \quad \text{where } \Lambdahat^{(t)} = \operatorname{diag}\Bigl\{\underbrace{\frac{\alpha^{(t)}}{r}, \ldots, \frac{\alpha^{(t)}}{r}}_{r \text{ times}},
        \underbrace{\frac{\beta}{N - r}, \ldots, \frac{\beta}{N - r}}_{N - r \text{ times}},\Bigr\},
    \end{equation}
    and the entries $\alpha^{(t)}$ are given by the recursion
    \begin{equation}
    \label{eq:discrete_time}
        \alpha^{(t+1)} = \alpha^{(t)} - \frac{b (\alpha^{(t)})^2}{(b + r) \alpha^{(t)} + r \beta}, \qquad \alpha^{(0)} = \alpha.
    \end{equation}
    The comparison \cref{eq:comparison} is valid with equality at time $t = 0$.
    Assume the comparison is valid at a time $t \geq 0$.
    Using the monotonicity of $\mat{\Phi}_b$ (\cref{lem:matrix}a) it follows that
    \begin{equation*}
        \mat{\Phi}_b^{(t+1)}(\Lambdahat)
        \preceq \mat{\Phi}_b(\Lambdahat^{(t)})
        = \phi_{\alpha^{(t)} + \beta}^{(b)}(\Lambdahat^{(t)}).
    \end{equation*}
    By \cref{lem:properties}e, the trailing $N - r$ diagonal entries of $\phi_{\alpha^{(t)} + \beta}^{(b)}(\Lambdahat^{(t)})$ are bounded by $\frac{\beta}{N-r}$ and the leading $r$ diagonal entries are bounded by
    \begin{equation*}
         \frac{1}{\frac{b}{\alpha^{(t)} + \beta} + \frac{r}{\alpha^{(t)}}}
         = \frac{1}{r} \Bigl[ \alpha^{(t)} - \frac{b (\alpha^{(t)})^2}{(b + r) \alpha^{(t)} + r \beta} \Bigr].
    \end{equation*}
    Therefore, \cref{eq:comparison} holds for all $t\ge 0$ by induction.

    At each instant $t = 0, 1, \ldots$, the discrete-time dynamical system $\alpha^{(t)}$ \cref{eq:discrete_time} is bounded from above by a continuous-time dynamical system $\alpha(t)$ that satisfies
    \begin{equation}
    \label{eq:continuous_time}
        \dot{\alpha}(t) = - \frac{b (\alpha(t))^2}{(b + r) \alpha(t) + r \beta}, \qquad \alpha^{(0)} = \alpha.
    \end{equation}
    The comparison holds because $x \mapsto b x^2 / ((b + r) x + r \beta)$ is a non-negative, non-decreasing function.

    Last, assume that $\alpha \leq \varepsilon \beta$.
    (Otherwise, the theorem holds trivially at all times $t \geq 0$.)
    By separation of variables, the continuous-time dynamical system \cref{eq:continuous_time} satisfies $\alpha(t) \leq \varepsilon \beta$ as soon as
    \begin{equation*}
        t = \Bigl[ -\frac{b+r}{b} \log(\alpha(t)) + \frac{r \beta}{b} \alpha(t)^{-1} \Bigr]^{\alpha(t) = \varepsilon \beta}_{\alpha(0) = \alpha}
        = \frac{b+r}{b} \log\Bigl(\frac{1}{\varepsilon \eta}\Bigr)
        + \frac{r}{b \varepsilon} - \frac{r}{b} \eta.
    \end{equation*}
    This establishes a slightly stronger version of the theorem, completing the proof.
\end{proof}

\section{Extension: Accelerated randomly pivoted \QR} \label{sec:qr}

The pivoted partial Cholesky decomposition of a psd matrix $\mat{A}$ is closely related to the pivoted partial \QR decomposition of a general, possibly rectangular, matrix $\mat{B} \in \complex^{M\times N}$.
This section reviews the connection betweeen algorithms and introduces a new \emph{accelerated randomly pivoted \QR algorithm} for general low-rank matrix approximation.

A low-rank approximation for a general matrix can be computed by column-pivoted \QR as follows:
Begin by initializing the approximation $\Bhat^{(0)} \coloneqq \mat{0}$ and the residual $\mat{B}^{(0)} \coloneqq \mat{B}$.
For each $i = 0,\ldots,k-1$, perform the following steps:
\begin{enumerate}
    \item \textbf{Select a pivot column.} Choose $s_{i+1} \in \{1,\ldots,N\}$.
    \item \textbf{Update.} Evaluate the column $\mat{B}^{(i)}(:, s_{i+1})$ indexed by the pivot index $s_{i+1}$.
    Update the approximation and the residual:
    \begin{align*}
        \Bhat^{(i+1)} \coloneqq \Bhat^{(i)} + \frac{\mat{B}^{(i)}(:, s_{i+1})\mat{B}^{(i)}(:, s_{i+1})^*\mat{B}^{(i)}}{\norm{\mat{B}^{(i)}(:, s_{i+1})}^2}, \\
        \mat{B}^{(i+1)} \coloneqq \mat{B}^{(i)} - \frac{\mat{B}^{(i)}(:, s_{i+1})\mat{B}^{(i)}(:, s_{i+1})^*\mat{B}^{(i)}}{\norm{\mat{B}^{(i)}(:, s_{i+1})}^2}.
    \end{align*}
\end{enumerate}
Each update has the effect of orthogonalizing the residual $\mat{B}^{(i)}$ against the column indexed by the selected pivot and adjusting the approximation $\Bhat^{(i)}$ accordingly.
This version of column-pivoted \QR has issues with numerically stability and is provided for conceptual understanding only.
Numerically stable implementation of \QR decomposition methods is discussed in numerical linear algebra textbooks, e.g., \cite[ch.~5]{GV13}.

Pivoted \QR and Cholesky approximations are closely related (see, e.g., \cite{Hig90a}):

\begin{proposition}[\QR and Cholesky] \label{prop:qr_cholesky}
    Suppose that column-pivoted \QR is executed on a matrix $\mat{B}$ with pivots $s_1,\ldots,s_k$ producing an approximation $\Bhat\approx \mat{B}$.
    Similarly, suppose that partial Cholesky is performed on the \emph{Gram matrix} $\mat{A} \coloneqq \mat{B}^*\mat{B}$ with the same pivots $s_1,\ldots,s_k$, producing an approximation $\Ahat \approx \mat{A}$.
    Then $\Bhat^*\Bhat = \Ahat$.
\end{proposition}

This result establishes the connection between the low-rank approximation of a psd matrix based on the pivoted Cholesky decomposition and the low-rank approximation of a general, rectangular matrix by the column-pivoted \QR decomposition.
In particular, any pivot selection strategy for Cholesky decomposition (i.e., random pivoting) has an analog for \QR decomposition and visa versa.

\begin{algorithm}[t]
  \caption{Randomly pivoted \QR}
  \label{alg:rpqr}
  \begin{algorithmic}
    \Require Matrix $\mat{B} \in \mathbb{C}^{M\times N}$; approximation rank $k$
    \Ensure Matrices $\mat{Q} \in \complex^{M\times k},\mat{F} \in \mathbb{C}^{N\times k}$ defining $\Bhat = \mat{Q}\mat{F}^*$; pivot set $\set{S}$
    \State Initialize $\mat{Q} \gets \mat{0}_{M\times k}$, $\mat{F} \leftarrow \mat{0}_{N\times k}$, $\set{S} \leftarrow \emptyset$, and $\vec{d} \leftarrow \Call{SquaredColumnNorms}{\mat{B}}$ 
    \For{$i = 1$ to $k$}
    \State Sample pivot $s \sim \vec{d} / \sum_j d_j$
    \State Induct new pivot $\set{S} \gets \set{S} \cup \{s\}$
    \State $\vec{g} \leftarrow \mat{B}(:,s) - \mat{Q}(:,1:1-i) \mat{F}(s,1:i-1)^*$
    \State $\mat{Q}(:,i) \leftarrow \vec{g} / \norm{\vec{g}}$
    \State $\mat{F}(:,i) \gets \mat{B}^*\mat{Q}(:,i)$
    \State $\vec{d} \leftarrow \vec{d} - |\mat{F}(:,i)|^2$
    \EndFor
  \end{algorithmic}
\end{algorithm}

\begin{algorithm}[t]
  \caption{Accelerated randomly pivoted \QR}
  \label{alg:accelerated_rpqr}
  \begin{algorithmic}
    \Require Matrix $\mat{A} \in \mathbb{C}^{M\times N}$; block size $b$; number of rounds $t$
    \Ensure Matrices $\mat{Q} \in \complex^{M\times k},\mat{F} \in \mathbb{C}^{N\times k}$ defining $\Bhat = \mat{Q}\mat{F}^*$; pivot set $\set{S}$
    \State Initialize $\mat{Q} \gets \mat{0}_{M\times 0}$, $\mat{F} \leftarrow \mat{0}_{N\times 0}$, $\set{S} \leftarrow \emptyset$, and $\vec{u} \leftarrow \Call{SquaredColumnNorms}{\mat{B}}$
    \For{$i = 0$ to $t - 1$}
        \State \underline{\textit{Step 1: Propose a block of pivots}} \vspace{0.5em}
        \State Sample $s_{ib+1}', \ldots, s_{(i+1)b}' \stackrel{\rm iid}{\sim} \vec{u}$ and set $\set{S}_i' \gets \{s_{ib+1}', \ldots, s_{(i+1)b}'\}$
        \State 
        \State \underline{\textit{Step 2: Rejection sampling}} \vspace{0.5em}
        \State $\mat{C} \gets \mat{B}(:, \set{S}_i') - \mat{Q}\mat{F}(\set{S}_i',:)^*$ 
        \State $\set{S}_i\gets \Call{\RejectChol}{\set{S}_i',\mat{C}^*\mat{C}}$ \Comment{Discard second output}
        \State $\set{S} \gets \set{S} \cup \set{S}_i$ \Comment{Update pivots} 
        \State 
        \State \underline{\textit{Step 3: Update low-rank approximation and proposal distribution}} \vspace{0.5em}
        \State $\mat{Q}_\perp \gets \mat{B}(:, \set{S}_i) - \mat{Q}\mat{F}(\set{S}_i,:)^*$
        \State $\mat{Q}_\perp \gets \mat{Q}_\perp - \mat{Q} (\mat{Q}^*\mat{Q}_\perp)$ \Comment{Extra step of Gram--Schmidt, for stability}
        \State $\mat{Q}_\perp \gets \Call{Orth}{\mat{Q}_\perp}$ \Comment{Orthonormalize columns}
        \State $\mat{Q}\gets \onebytwo{\mat{Q}}{\mat{Q}_\perp}$
        \State $\mat{G} \gets \mat{B}^* \mat{Q}_\perp$
        \State $\mat{F} \gets \onebytwo{\mat{F}}{\mat{G}}$
        \State $\vec{u} \gets \vec{u} - \Call{SquaredRowNorms}{\mat{G}}$ \Comment{Update diagonal}
        \State $\vec{u} \gets \max \{\vec{u},\vec{0}\}$ \Comment{Helpful in floating point arithmetic}
    \EndFor
  \end{algorithmic}
\end{algorithm}

The analog of \RPCholesky is the randomly pivoted \QR algorithm, which was introduced by Deshpande, Vempala, Rademacher, and Wang \cite{DRVW06,DV06} under the name ``adaptive sampling''.
See \cref{alg:rpqr} for randomly pivoted \QR pseudocode.
As with \RPCholesky, randomly pivoted \QR can be extended using block and accelerated variants; see \cite{CETW23} for block randomly pivoted \QR pseudocode and see \cref{alg:accelerated_rpqr} for accelerated randomly pivoted \QR pseudocode.
Note that \cref{alg:accelerated_rpqr} uses a block Gram--Schmidt procedure to perform orthogonalization, which employs two orthogonalization steps for improved stability \cite{GLRE05}.
A different, even more numerically stable approach uses Householder reflectors instead \cite[Ch.~5]{GV13}.

Due to the connection between \QR and Cholesky (\cref{prop:qr_cholesky}), theoretical results for block and accelerated \RPCholesky immediately extend to their \QR analogs:

\begin{corollary}[Sufficient iterations for accelerated and block randomly pivoted \QR] \label{cor:qr}
    Consider a rectangular matrix $\mat{B} \in \mathbb{C}^{M \times N}$, and let $\mat{B}^{(t)}$ denote the random residual after applying $t$ rounds of accelerated randomly pivoted \QR or block randomly pivoted \QR with a block size $b \geq 1$.
    Then,
    \begin{equation*}
        \mathbb{E}\lVert \mat{B}^{(t)} \rVert_{\rm F}^2 \leq (1 + \varepsilon) \cdot \sum\nolimits_{i = r + 1}^N \sigma_i(\mat{B})^2
    \end{equation*}
    as soon as
    \begin{equation*}
        tb \geq \frac{r}{\varepsilon} + (r + b) \log\Bigl(\frac{1}{\epsilon \eta}\Bigr),
        \quad \text{where } \eta = \frac{\sum_{i = r + 1}^N \sigma_i(\mat{B})^2}{\sum_{i = 1}^r \sigma_i(\mat{B})^2}.
    \end{equation*}
\end{corollary}
Accelerated randomly pivoted \QR is a new algorithm, while block randomly pivoted \QR was introduced and analyzed in \cite[Thm.~1.2]{DRVW06}.
This earlier analysis has the large block size limitation that is discussed in \cref{sec:theory}.

\section*{Acknowledgements}
We thank Vivek Bharadwaj, Chao Chen, and Yijun Dong for helpful discussions.

\appendix

\section{Why blocking?\nopunct} \label{sec:why-blocking}
\Cref{fig:generation,fig:initial,fig:performance} demonstrate the empirical speed-ups of algorithms that simultaneously process large blocks of columns from kernel matrices.
\Cref{sec:kernels-submatrix} has discussed one reason for this speed-up, due to data movement.
This appendix describes two additional reasons why blocked (kernel) matrix algorithms can be faster: the fast Euclidean distances trick (\cref{sec:euclidean}) and matrix--matrix operations (\cref{sec:mat-mat}).

\subsection{The fast Euclidean distances trick} \label{sec:euclidean}

The fast Euclidean distances trick is a method for evaluating kernel matrix entries that take the special form
\begin{equation*}
    \kappa(\vec{x}_i,\vec{x}_j) = \phi(\norm{\vec{x}_i - \vec{x}_j}) \quad \text{for } i \in \set{S}_1, j\in\set{S}_2,
\end{equation*}
where $\phi: \mathbb{R} \rightarrow \mathbb{R}$ is a univariate function.
The trick relies on decomposing the square Euclidean distance into three summands:
\begin{equation} \label{eq:distances}
    \norm{\vec{x}_i - \vec{x}_j}^2 = \norm{\vec{x}_i}^2 - 2 \Re\{\vec{x}_i^* \vec{x}_j^{\vphantom{\top}}\} + \norm{\vec{x}_j}^2.
\end{equation}
It takes just $\order(d |\set{S}_1| + d |\set{S}_2|)$ operations to compute the square norms $\norm{\vec{x}_i}^2$ and $\norm{\vec{x}_j}^2$ for each $i \in \set{S}_1$ and $j \in \set{S}_2$.
Meanwhile, the inner products $\vec{x}_i^* \vec{x}_j^{\vphantom{*}}$ can be obtained from the entries of the inner product matrix
\begin{equation} \label{eq:fedt_multiply}
     \mat{X}(:, \set{S}_1)^* \mat{X}(:, \set{S}_2) \quad \text{where} \quad 
     \mat{X} = \begin{bmatrix}
        \vec{x}_1 & \cdots & \vec{x}_N
    \end{bmatrix}.
\end{equation}
Computing the inner product matrix requires $\mathcal{O}(d |\set{S}_1| |\set{S}_2|)$ operations using a single matrix--matrix, a highly optimized operation on modern computers \cite[Ch.~20.1]{DFF+03}.
The fast Euclidean distances trick explains why generating columns can be $10\times$ faster for a high-dimensional Gaussian kernel (\cref{fig:generation}, left) than an $\ell_1$ Laplace kernel (\cref{fig:generation}, right).

As a potential concern, the fast Euclidean distances trick can lead to \emph{catastrophic cancellations} \cite[sec.~1.7]{Hig02} in floating point arithmetic when two points $\vec{x}_i$ and $\vec{x}_j$ are close.
Catastrophic cancellations sometimes produce inaccurate results for block \RPCholesky, even after adding a multiple of the machine precision to the diagonal of the target matrix to promote stability.
In contrast, accelerated \RPCholesky avoids catastrophic cancellations by rejecting the small diagonal values as pivots.
The opportunity to use the fast Euclidean distances trick without catastrophic cancellations is a major advantage of accelerated \RPCholesky.
Nonetheless, to promote a simple comparison, all the experiments involving block \RPCholesky avoid the fast Euclidean distances trick; the code computes distances by subtracting $\vec{x}_i - \vec{x}_j$ and directly calculating the norm.

\subsection{Matrix--matrix operations}\label{sec:mat-mat}
The cost of \RPCholesky-type methods can be decomposed into the cost of generating columns of the kernel matrix and the cost of processing the columns to execute the algorithm.
Blocking leads to speed-ups in both steps of the algorithm.
For the second step, blocking organizes the operations into efficient matrix--matrix operations.
Accelerated \RPCholesky takes advantage of these optimized operations for an improved runtime.
Simple \RPCholesky, by contrast, uses only vector--vector and matrix--vector operations, which are slower.

\section{Deferred proofs from \cref{sec:theory}} \label{sec:deferred-proofs}
This section contains deferred proofs from \cref{sec:theory}.

\begin{proof}[Proof of \cref{lem:properties}]
The proof is based on a series of direct calculations.

\emph{\textbf{Parts (a)--(b).}} Observe that 
\begin{equation}
\label{eq:concave}
    \phi_{\alpha}(x) = \tfrac{\alpha}{4} - \tfrac{1}{\alpha}\bigl(x - \tfrac{\alpha}{2}\bigr)^2
\end{equation}
is a concave quadratic which is uniformly bounded by $\frac{\alpha}{4}$.
This immediately establishes part (a) and part (b) in the case $b = 1$.
Next assume part (b) is valid for some $b \geq 1$ and assume $x\le x' \le \alpha/2$.
By the boundedness property (a), $\phi_{\alpha}^{(b)}(x) \leq \phi_{\alpha}^{(b)}(x') \leq \alpha/4$.
Since $x \mapsto \phi_{\alpha}(x)$ is non-decreasing over the range $[0, \frac{\alpha}{2}]$, it follows:
\begin{equation*}
    \phi_{\alpha}^{(b+1)}(x) = \phi_{\alpha}(\phi_{\alpha}^{(b)}(x)) \leq \phi_{\alpha}(\phi_{\alpha}^{(b)}(x')) = \phi_{\alpha}^{(b+1)}(x'),
\end{equation*}
which establishes part (b) by induction.

\textbf{\emph{Part (c).}} Calculate
\begin{equation*}
    \phi_{\beta \cdot \alpha}(\beta \cdot x) 
    = \beta x - \tfrac{\beta}{\alpha} x^2 
    = \beta \bigl[ x - \tfrac{x^2}{\alpha} \bigr] = \beta \cdot \phi_{\alpha}(x),
\end{equation*}
which confirms part (c) when $b = 1$.
Homogeneity for $b \geq 1$ follows by induction.

\textbf{\emph{Part (d).}} The quadratic representation \cref{eq:concave} shows that $\phi_1$ is concave.
Assume that $\phi^{(b)}_1$ is concave for some $b \geq 1$.
Then, calculate for $\theta \in (0, 1)$:
\begin{align*}
    \theta \phi_1^{(b+1)}(x) + (1 - \theta) \phi_1^{(b+1)}(y)
    &= \theta \phi_1 (\phi_1^{(b)}(x)) + (1 - \theta) \phi_1(\phi_1^{(b)}(y)) \\
    &\leq \phi_1(\theta \phi_1^{(b)}(x) + (1 - \theta) \phi_1^{(b)}(y)) \\
    &\leq \phi_1(\phi_1^{(b)}(\theta x + (1 - \theta) y)) = \phi_1^{(b+1)}(\theta x + (1 - \theta) y).
\end{align*}
The first inequality is concavity of $\phi_1$, and the second inequality is concavity of $\phi_1^{(b)}$ and parts (a)--(b).
Conclude by induction that $\phi^{(b)}_1$ is concave for each $b \geq 1$.

The next step is to establish joint concavity of $(x,\alpha) \mapsto \phi_\alpha^{(b)}(x)$.
For any $\theta \in (0, 1)$:
\begin{align*}
    \theta \phi_{\alpha_1}^{(b)}(x_1)
    + (1 - \theta) \phi_{\alpha_2}^{(b)}(x_2)
    &= \theta \alpha_1 \phi_1^{(b)}\Bigl(\frac{x_1}{\alpha_1}\Bigr)
    + (1 - \theta) \alpha_2 \phi_1^{(b)}\Bigl(\frac{x_2}{\alpha_2}\Bigr) \\
    &\leq \bigl[\theta \alpha_1 + (1 - \theta) \alpha_2\bigr] \phi_1^{(b)} \Bigl(\frac{\theta x_1 + (1 - \theta) x_2}{\theta \alpha_1 + (1 - \theta) \alpha_2}\Bigr) \\
    &= \phi^{(b)}_{\theta \alpha_1 + (1 - \theta) \alpha_2}(\theta x_1 + (1 - \theta) x_2).
\end{align*}
The first equality is part (c), and the inequality is the concavity of $\phi_1^{(b)}$.
Part (d) is proven.

\textbf{\emph{Part (e).}}
Calculate
\begin{equation*}
    \phi_{\alpha}(x) = x - \frac{x^2}{\alpha} 
    \leq  x - \frac{x^2}{x + \alpha}
    = \frac{\alpha x}{x + \alpha} 
    = \frac{1}{\alpha^{-1} + x^{-1}},
\end{equation*}
which confirms part (e) in the case $b = 1$.
Next assume part (e) holds for some value $b \geq 1$.
Then using the monotonicity property (b),
\begin{align*}
    \phi_{\alpha}^{(b+1)}(x) 
    &= \phi_{\alpha}(\phi_{\alpha}^{(b)}(x)) \\
    &\leq \phi_{\alpha}\Bigl(\frac{1}{b \alpha^{-1} + x^{-1}}\Bigr) = \frac{1}{b \alpha^{-1} + x^{-1}} - \frac{1}{(b \alpha^{-1} + x^{-1}) (b + \alpha x^{-1})} \\
    &\leq \frac{1}{b \alpha^{-1} + x^{-1}} - \frac{1}{(b \alpha^{-1} + x^{-1})(b + 1 + \alpha x^{-1})} = \frac{1}{(b + 1) \alpha^{-1} + x^{-1}}.
\end{align*}
Therefore, part (e) holds by induction.
\end{proof}

\begin{proof}[Proof of \cref{lem:matrix}]
The proof is based on a series of direct calculations.

\textbf{\emph{Part (a).}} Assume $ \mat{\Lambda}_1 \preceq \mat{\Lambda}_2$.
Use the positive homogeneity and concavity of $(x, \alpha) \mapsto \phi^{(b)}_{\alpha}(x)$ (\cref{lem:properties}c--d) to calculate
    \begin{align*}
        \tfrac{1}{2} \phi^{(b)}_{\tr(\mat{\Lambda}_1)}(\mat{\Lambda}_1)
        &\preceq \tfrac{1}{2} \phi^{(b)}_{\tr(\mat{\Lambda}_1)}(\mat{\Lambda}_1)
        + \tfrac{1}{2} \phi^{(b)}_{\tr(\mat{\Lambda}_2 - \mat{\Lambda}_1)}(\mat{\Lambda}_2 - \mat{\Lambda}_1) \\
        &\preceq \phi^{(b)}_{\frac{1}{2} \tr(\mat{\Lambda}_2)}(\tfrac{1}{2} \mat{\Lambda}_2)
        = \tfrac{1}{2} \phi^{(b)}_{\tr(\mat{\Lambda}_2)}(\mat{\Lambda}_2).
    \end{align*}
    Equivalently, $\mat{\Phi}_b(\mat{\Lambda}_1) \preceq \mat{\Phi}_b(\mat{\Lambda}_2)$.
    Iterating this property establishes part (a).

    \textbf{\emph{Part (b).}}
    Use the concavity of $(x, \alpha) \mapsto \phi^{(b)}_{\alpha}(x)$ (\cref{lem:properties}d) to conclude:
    \begin{equation*}
        \theta \mat{\Phi}_b(\mat{\Lambda}_1) 
        + (1 - \theta) \mat{\Phi}_b(\mat{\Lambda}_2)
        \preceq
        \mat{\Phi}_b( \theta \mat{\Lambda}_1 + (1-\theta) \mat{\Lambda}_2)  \quad \text{for } \theta \in (0, 1).
    \end{equation*}
    This confirms part (b) in the case $t = 1$.
    Next assume part (b) holds for $t \geq 1$.
    Then
    \begin{align*}
        \theta \mat{\Phi}_b^{(t+1)}(\mat{\Lambda}_1) 
        + (1 - \theta) \mat{\Phi}_b^{(t+1)}(\mat{\Lambda}_2)
        &= \theta \mat{\Phi}_b(\mat{\Phi}_b^{(t)}(\mat{\Lambda}_1)) 
        + (1 - \theta) \mat{\Phi}_b(\mat{\Phi}_b^{(t)}(\mat{\Lambda}_2)) \\
        &\preceq \mat{\Phi}_b(\theta \mat{\Phi}_b^{(t)}(\mat{\Lambda}_1) + (1 - \theta) \mat{\Phi}_b^{(t)}(\mat{\Lambda}_2)) \\
        &\preceq \mat{\Phi}_b(\mat{\Phi}_b^{(t)}(\theta \mat{\Lambda}_1 + (1 - \theta) \mat{\Lambda}_2)) \\
        &=
        \mat{\Phi}_b^{(t+1)}( \theta \mat{\Lambda}_1 + (1-\theta) \mat{\Lambda}_2).
    \end{align*}
    The first inequality is the concavity of $\mat{\Phi}_b$, and the second inequality is the concavity of $\mat{\Phi}_b^{(t)}$ and the monotonicity property (a).
    Therefore, by induction, part (b) holds for $t\ge 1$.

    \textbf{\emph{Part (c).}} Calculate
    \begin{align*}
        & \mat{\Phi}_1( \theta \mat{A}_1 + (1-\theta) \mat{A}_2)
        - \theta \mat{\Phi}_1(\mat{A}_1)
        - (1 - \theta) \mat{\Phi}_1(\mat{A}_2) \\
        &= \frac{\theta \mat{A}_1^2}{\tr(\mat{A}_1)} 
        + \frac{(1-\theta) \mat{A}_2^2}{\tr(\mat{A}_2)}
        - \frac{[\theta \mat{A}_1 + (1 - \theta) \mat{A}_2]^2}{\theta \tr(\mat{A}_1) + (1 - \theta) \tr(\mat{A}_2)} \\
        &= \frac{\theta(1 - \theta) \tr(\mat{A}_1) \tr(\mat{A}_2)}{\theta \tr(\mat{A}_1) + (1 - \theta) \tr(\mat{A}_2)}
        \biggl[\frac{\mat{A}_1}{\tr(\mat{A}_1)}
        - \frac{\mat{A}_2}{\tr(\mat{A}_2)}
        \biggr]^2
        \succeq \mat{0}.
    \end{align*}
    This establishes the concavity of $\mat{\Phi}_1$ over all psd matrices.
\end{proof}

\section{Robust blockwise random pivoting} \label{sec:rbrp}
The robust blockwise random pivoting (RBRP) method was introduced by Dong, Chen, Martinsson, and Pearce to improve the performance of block randomly pivoted \QR \cite{DCMP23}.
Similar to accelerated randomly pivoted \QR,
each round of RBRP begins by generating a block of proposed pivots $\set{S}'_i = \{s_{ib+1}',\ldots,s_{(i+1)b}' \}$.
Then RBRP chooses a subset of the pivots by a deterministic, greedy method.
In detail, let $\mat{G}$ denote the residual submatrix with columns indexed by $\set{S}'_i$:
\begin{equation*}
    \mat{G} = \mat{B}(:, \set{S}'_i) - \Bhat(:, \set{S}'_i).
\end{equation*}
RBRP applies a partial \QR decomposition to $\mat{G}$ using greedy column pivoting, stopping when the squared Frobenius norm of $\mat{G}$ has been reduced by a factor $b$.
RBRP then accepts the pivots chosen before the tolerance is met and rejects all others.

\begin{table}[t]
    \centering
    \begin{tabular}{ccc}
        \toprule
         & Relative trace error & Runtime (sec) \\
        \midrule
        Block \RPCholesky & 3.89e-04 $\pm$ 1.40e-04 & 7.00 $\pm$ 0.63  \\
        Accelerated \RPCholesky	& 4.85e-07 $\pm$ 3.60e-08	& 7.43 $\pm$ 1.06 \\
        RBRP Cholesky &	6.62e-07 $\pm$ 1.35e-07 &	7.53 $\pm$ 0.80 \\
        \bottomrule
    \end{tabular}
    \caption{Runtime and relative error for rank-$1000$ approximation of the smile example from \cref{fig:initial} with block size $b=120$.
    Here the mean and standard deviation are computed empirically over $100$ independent trials.}
    \label{tab:rbrp}
\end{table}

RBRP can be extended to psd low-rank approximation, resulting in a \textit{robust block randomly pivoted Cholesky} method.
A comparison of RBRP Cholesky, block \RPCholesky, and accelerated \RPCholesky appears in \cref{tab:rbrp}.
Both accelerated \RPCholesky and RBRP Cholesky achieve dramatically lower trace error than standard block \RPCholesky on this example, demonstrating that both methods are effective at filtering redundant columns.
This example does suggest a few preliminary reasons to (slightly) prefer accelerated \RPCholesky over RBRP Cholesky, though.
First, RBRP Cholesky has $1.4\times$ higher trace error than accelerated \RPCholesky on this example, and the error of the output is more variable for RBRP Cholesky.
Second, RBRP Cholesky is slightly slower than accelerated \RPCholesky in our implementation.
Finally, rigorous error bounds are known for accelerated \RPCholesky (and \QR), but not known for RBRP methods.
Notwithstanding these difference, RBRP methods and accelerated \RPCholesky/\QR both appear to be fast and accurate in practice, and both methods significantly out-compete standard block \RPCholesky/\QR. 

\section{Existing algorithms}
This section provides pseudocode for \RPCholesky (\cref{alg:rpcholesky}) and block \RPCholesky (\cref{alg:block_rpc}).

\begin{algorithm}[H]
  \caption{\RPCholesky}
  \label{alg:rpcholesky}
  \begin{algorithmic}
    \Require Psd matrix $\mat{A} \in \mathbb{C}^{N\times N}$; approximation rank $k$
    \Ensure Matrix $\mat{F} \in \mathbb{C}^{N\times k}$ defining low-rank approximation $\Ahat = \mat{F}\mat{F}^*$; pivot set $\set{S}$
    \State Initialize $\mat{F} \leftarrow \mat{0}_{N\times k}$ and $\vec{d} \leftarrow \diag \mat{A}$ 
    \For{$i = 1$ to $k$}
    \State Sample pivot $s \sim \vec{d} / \sum_j d_j$
    \State Induct new pivot $\set{S} \gets \set{S} \cup \{s\}$
    \State $\vec{g} \leftarrow \mat{A}(:,s) - \mat{F}(:,1:i-1) \mat{F}(s,1:i-1)^*$
    \State $\mat{F}(:,i) \leftarrow \vec{g} / \sqrt{g_s}$
    \State $\vec{d} \leftarrow \vec{d} - |\mat{F}(:,i)|^2$
    \EndFor
  \end{algorithmic}
\end{algorithm}

\begin{algorithm}[H]
  \caption{Block \RPCholesky} \label{alg:block_rpc}
  \begin{algorithmic}
	    \Require Psd matrix $\mat{A} \in \mathbb{C}^{N\times N}$; block size $b$; number of rounds $t$
	    \Ensure Matrix $\mat{F} \in \mathbb{C}^{N \times bt}$ defining low-rank approximation $\Ahat = \mat{F}\mat{F}^*$; pivot set $\set{S}$ 
	    \State Initialize $\mat{F} \leftarrow \mat{0}_{N \times bk}$,
	    $\set{S} \leftarrow \emptyset$, and
	    $\vec{d} \leftarrow \diag \mat{A}$
	    \For{$i = 0$ to $t - 1$}
	    \State Sample $s_{ib + 1}, \ldots, s_{ib + b} \stackrel{\rm iid}{\sim} \vec{d} / \sum_j d_j$
	    \State $\set{S}' \leftarrow \Call{Unique}{\{s_{ib + 1}, \ldots, s_{ib + T}\}}$
	    \State $\set{S} \leftarrow \set{S} \cup \set{S}'$
	    \State $\mat{G} \leftarrow \mat{A}(:,\set{S}') - \mat{F}(:,1:iT) \mat{F}(\set{S}',1:iT)^*$
	    \State $\mat{R} \leftarrow \Call{Chol}{\mat{G}(\set{S}',:)}$ 
	    \State $\mat{F}(:,ib + 1:ib + |\set{S}'|) \leftarrow \mat{G} \mat{R}^{-1}$
	    \State $\vec{d} \leftarrow \vec{d} - \Call{SquaredRowNorms}{\mat{F}(:,ib+1:ib + b)}$
	    \EndFor
	    \State  Remove zero columns from $\mat{F}$
	  \end{algorithmic}
\end{algorithm}

\bibliographystyle{siamplain}
\bibliography{refs}

\begin{thebibliography}{10}

\bibitem{alaoui2015fast}
{\sc A.~Alaoui and M.~W. Mahoney}, {\em Fast randomized kernel ridge regression
  with statistical guarantees}, in Proceedings of the 28th International
  Conference on Neural Information Processing Systems, 2015,
  \url{https://dl.acm.org/doi/10.5555/2969239.2969326}.

\bibitem{aristoff2024fast}
{\sc D.~Aristoff, M.~Johnson, G.~Simpson, and R.~J. Webber}, {\em The fast
  committor machine: {I}nterpretable prediction with kernels}, The Journal of
  Chemical Physics, 161 (2024), p.~084113,
  \url{https://doi.org/10.1063/5.0222798}.

\bibitem{avron2017faster}
{\sc H.~Avron, K.~L. Clarkson, and D.~P. Woodruff}, {\em Faster kernel ridge
  regression using sketching and preconditioning}, SIAM Journal on Matrix
  Analysis and Applications, 38 (2017), pp.~1116--1138,
  \url{https://doi.org/10.1137/16M1105396}.

\bibitem{CETW23}
{\sc Y.~Chen, E.~N. Epperly, J.~A. Tropp, and R.~J. Webber}, {\em Randomly
  pivoted {{Cholesky}}: {{Practical}} approximation of a kernel matrix with few
  entry evaluations}, arXiv preprint arXiv:2207.06503,  (2023),
  \url{https://arxiv.org/abs/2207.06503v5}.

\bibitem{chmiela2018towards}
{\sc S.~Chmiela, H.~E. Sauceda, K.-R. Müller, and A.~Tkatchenko}, {\em Towards
  exact molecular dynamics simulations with machine-learned force fields},
  Nature Communications, 9 (2018),
  \url{https://doi.org/10.1038/s41467-018-06169-2}.

\bibitem{chmiela2017machine}
{\sc S.~Chmiela, A.~Tkatchenko, H.~E. Sauceda, I.~Poltavsky, K.~T. Sch{\"u}tt,
  and K.-R. M{\"u}ller}, {\em Machine learning of accurate energy-conserving
  molecular force fields}, Science Advances, 3 (2017), p.~e1603015,
  \url{https://doi.org/10.1126/sciadv.1603015}.

\bibitem{chmiela2023accurate}
{\sc S.~Chmiela, V.~Vassilev-Galindo, O.~T. Unke, A.~Kabylda, H.~E. Sauceda,
  A.~Tkatchenko, and K.-R. Müller}, {\em Accurate global machine learning
  force fields for molecules with hundreds of atoms}, Science Advances, 9
  (2023), p.~eadf0873, \url{https://doi.org/10.1126/sciadv.adf0873}.

\bibitem{cutajar2016preconditioning}
{\sc K.~Cutajar, M.~Osborne, J.~Cunningham, and M.~Filippone}, {\em
  Preconditioning kernel matrices}, in Proceedings of The 33rd International
  Conference on Machine Learning, 2016,
  \url{https://proceedings.mlr.press/v48/cutajar16.html}.

\bibitem{DRVW06}
{\sc A.~Deshpande, L.~Rademacher, S.~S. Vempala, and G.~Wang}, {\em Matrix
  approximation and projective clustering via volume sampling}, Theory of
  Computing, 2 (2006), pp.~225--247,
  \url{https://doi.org/10.4086/toc.2006.v002a012}.

\bibitem{DV06}
{\sc A.~Deshpande and S.~Vempala}, {\em Adaptive sampling and fast low-rank
  matrix approximation}, in Proceedings of the 9th International Conference on
  Approximation Algorithms for Combinatorial Optimization Problems, and 10th
  International Conference on Randomization and Computation, 2006,
  \url{https://doi.org/10.1007/11830924_28}.

\bibitem{DEF+23}
{\sc M.~D{\'i}az, E.~N. Epperly, Z.~Frangella, J.~A. Tropp, and R.~J. Webber},
  {\em Robust, randomized preconditioning for kernel ridge regression}, arXiv
  preprint arXiv:2304.12465,  (2024), \url{https://arxiv.org/abs/2304.12465v4}.

\bibitem{DCMP23}
{\sc Y.~Dong, C.~Chen, P.-G. Martinsson, and K.~Pearce}, {\em Robust blockwise
  random pivoting: {{Fast}} and accurate adaptive interpolative decomposition},
  arXiv preprint arXiv:2309.16002,  (2023),
  \url{https://arxiv.org/abs/2309.16002v3}.

\bibitem{DFF+03}
{\sc J.~Dongarra, I.~Foster, G.~Fox, W.~Gropp, K.~Kennedy, L.~Torczon, and
  A.~White}, eds., {\em Sourcebook of Parallel Computing}, Morgan Kaufmann
  Publishers Inc., San Francisco, CA, USA, 2003.

\bibitem{DM05}
{\sc P.~Drineas and M.~W. Mahoney}, {\em On the {N}ystr{\"o}m method for
  approximating a {G}ram matrix for improved kernel-based learning}, Journal of
  Machine Learning Research, 6 (2005), pp.~2153--2175,
  \url{http://jmlr.org/papers/v6/drineas05a.html}.

\bibitem{EM23}
{\sc E.~N. Epperly and E.~Moreno}, {\em Kernel quadrature with randomly pivoted
  {C}holesky}, in Proceedings of the 37th International Conference on Neural
  Information Processing Systems, 2024,
  \url{https://dl.acm.org/doi/10.5555/3666122.3668997}.

\bibitem{FS01}
{\sc S.~Fine and K.~Scheinberg}, {\em Efficient {SVM} training using low-rank
  kernel representations}, Journal of Machine Learning Research, 2 (2002),
  p.~243–264, \url{https://dl.acm.org/doi/10.5555/944790.944812}.

\bibitem{frangella2023randomized}
{\sc Z.~Frangella, J.~A. Tropp, and M.~Udell}, {\em Randomized {{Nystr{\"o}m}}
  preconditioning}, SIAM Journal on Matrix Analysis and Applications, 44
  (2023), pp.~718--752, \url{https://doi.org/10.1137/21M1466244}.

\bibitem{gardner2018gpytorch}
{\sc J.~Gardner, G.~Pleiss, K.~Q. Weinberger, D.~Bindel, and A.~G. Wilson},
  {\em {GPyTorch}: {Blackbox} matrix-matrix {Gaussian} process inference with
  {GPU} acceleration}, in Proceedings of the 32nd International Conference on
  Neural Information Processing Systems, 2018,
  \url{https://dl.acm.org/doi/10.5555/3327757.3327857}.

\bibitem{GLRE05}
{\sc L.~Giraud, J.~Langou, M.~Rozlo{\v z}n{\'i}k, and J.~V.~D. Eshof}, {\em
  Rounding error analysis of the classical {{Gram-Schmidt}} orthogonalization
  process}, Numerische Mathematik, 101 (2005), pp.~87--100,
  \url{https://doi.org/10.1007/s00211-005-0615-4}.

\bibitem{GV13}
{\sc G.~H. Golub and C.~F. Van~Loan}, {\em Matrix Computations}, Johns Hopkins
  Studies in the Mathematical Sciences, Johns Hopkins Press, Baltimore,
  fourth~ed., 2013.

\bibitem{Hig90a}
{\sc N.~J. Higham}, {\em Analysis of the {{Cholesky}} decomposition of a
  semi-definite matrix}, in Reliable {{Numerical Commputation}}, M.~G. Cox and
  S.~Hammarling, eds., Oxford University Press, Sept. 1990,
  \url{https://doi.org/10.1093/oso/9780198535645.003.0010}.

\bibitem{Hig02}
{\sc N.~J. Higham}, {\em {{Accuracy and Stability of Numerical Algorithms}}},
  {Society for Industrial and Applied Mathematics}, {Philadelphia}, 2nd~ed.,
  2002.

\bibitem{kanagawa2018gaussianprocesseskernelmethods}
{\sc M.~Kanagawa, P.~Hennig, D.~Sejdinovic, and B.~K. Sriperumbudur}, {\em
  Gaussian processes and kernel methods: {A} review on connections and
  equivalences}, arXiv preprint arXiv:1807.02582,  (2018),
  \url{https://arxiv.org/abs/1807.02582v1}.

\bibitem{liu2004monte}
{\sc J.~S. Liu}, {\em Monte Carlo Strategies in Scientific Computing}, Springer
  New York, 2004, \url{https://doi.org/10.1007/978-0-387-76371-2}.

\bibitem{MCRR20}
{\sc G.~Meanti, L.~Carratino, L.~Rosasco, and A.~Rudi}, {\em Kernel methods
  through the roof: {H}andling billions of points efficiently}, in Proceedings
  of the 34th International Conference on Neural Information Processing
  Systems, 2020, \url{https://dl.acm.org/doi/abs/10.5555/3495724.3496932}.

\bibitem{MDM+22}
{\sc R.~Murray, J.~Demmel, M.~W. Mahoney, N.~B. Erichson, M.~Melnichenko, O.~A.
  Malik, L.~Grigori, P.~Luszczek, M.~Dereziński, M.~E. Lopes, T.~Liang,
  H.~Luo, and J.~Dongarra}, {\em Randomized numerical linear algebra : A
  perspective on the field with an eye to software}, arXiv preprint
  arXiv:2302.11474,  (2023), \url{https://arxiv.org/abs/2302.11474v2}.

\bibitem{MM17}
{\sc C.~Musco and C.~Musco}, {\em Recursive sampling for the {N}ystr\"{o}m
  method}, in Proceedings of the 31st International Conference on Neural
  Information Processing Systems, 2017,
  \url{https://dl.acm.org/doi/10.5555/3294996.3295140}.

\bibitem{MW17}
{\sc C.~Musco and D.~P. Woodruff}, {\em Sublinear time low-rank approximation
  of positive semidefinite matrices}, in 2017 IEEE 58th Annual Symposium on
  Foundations of Computer Science, 2017,
  \url{https://doi.org/10.1109/FOCS.2017.68}.

\bibitem{RO19}
{\sc A.~Rezaei and S.~O. Gharan}, {\em A polynomial time {MCMC} method for
  sampling from continuous determinantal point processes}, in Proceedings of
  the 36th International Conference on Machine Learning, 2019,
  \url{https://proceedings.mlr.press/v97/rezaei19a.html}.

\bibitem{rudi2018fast}
{\sc A.~Rudi, D.~Calandriello, L.~Carratino, and L.~Rosasco}, {\em On fast
  leverage score sampling and optimal learning}, in Proceedings of the 32nd
  International Conference on Neural Information Processing Systems, vol.~31,
  2018, \url{https://dl.acm.org/doi/10.5555/3327345.3327470}.

\bibitem{SS02}
{\sc B.~Sch\"olkopf and A.~J. Smola}, {\em {{Learning with Kernels: Support
  Vector Machines, Regularization, Optimization, and Beyond}}}, The MIT Press,
  Cambridge, MA, 2018.

\bibitem{TW23}
{\sc J.~A. Tropp and R.~J. Webber}, {\em Randomized algorithms for low-rank
  matrix approximation: {{Design}}, analysis, and applications}, arXiv preprint
  arXiv:2306.12418,  (2023), \url{https://arxiv.org/abs/2306.12418v3}.

\bibitem{unke2021machine}
{\sc O.~T. Unke, S.~Chmiela, H.~E. Sauceda, M.~Gastegger, I.~Poltavsky, K.~T.
  Sch{\"u}tt, A.~Tkatchenko, and K.-R. M{\"u}ller}, {\em Machine learning force
  fields}, Chemical Reviews, 121 (2021), pp.~10142--10186,
  \url{https://doi.org/10.1021/acs.chemrev.0c01111}.

\bibitem{wang2019exact}
{\sc K.~Wang, G.~Pleiss, J.~Gardner, S.~Tyree, K.~Q. Weinberger, and A.~G.
  Wilson}, {\em Exact {Gaussian} processes on a million data points}, in
  Proceedings of the 33rd International Conference on Neural Information
  Processing Systems, 2019,
  \url{https://dl.acm.org/doi/10.5555/3454287.3455599}.

\bibitem{WS00}
{\sc C.~K.~I. Williams and M.~Seeger}, {\em Using the {N}ystr\"{o}m method to
  speed up kernel machines}, in Proceedings of the 13th International
  Conference on Neural Information Processing Systems, 2000,
  \url{https://dl.acm.org/doi/10.5555/3008751.3008847}.

\bibitem{Woo14a}
{\sc D.~P. Woodruff}, {\em Sketching as a {{Tool}} for {{Numerical Linear
  Algebra}}}, Foundations and Trends in Theoretical Computer Science, 10
  (2014), pp.~1--157, \url{https://doi.org/10.1561/0400000060}.

\bibitem{Zha11}
{\sc F.~Zhang}, {\em {{Matrix Theory: Basic Results and Techniques}}},
  Springer, New York, 2nd~ed., 2011.

\end{thebibliography}
	
\end{document}